\documentclass{amsart}
\usepackage{tikz}

\usepackage{amsmath,amsfonts,amsthm,amssymb,amscd, verbatim, graphicx,textcomp, hyperref}
\usepackage[T1]{fontenc}
\usepackage[utf8]{inputenc}
\usepackage[margin=1.5in]{geometry}
\usepackage{graphics,graphicx}
\usepackage{pstricks,pst-node,pst-tree}

\usepackage{pinlabel}
\usepackage{lscape}
\usepackage{latexsym}
\usepackage{multirow}

\usepackage{booktabs}

\newtheorem{Thm}{Theorem}[section]
\newtheorem{Cor}[Thm]{Corollary}
\newtheorem{Conj}[Thm]{Conjecture}
\newtheorem{Prop}[Thm]{Proposition}

\newtheorem{Lem}[Thm]{Lemma}
\newtheorem{Rem}[Thm]{Remark}
\newtheorem{Prob}[Thm]{Problem}
\newtheorem{Const}[Thm]{Construction}

\theoremstyle{definition}

\newtheorem{Ex}[Thm]{Example}

\theoremstyle{remark}

\numberwithin{equation}{section}

\newcommand{\Aut}{\operatorname{Aut}}

\renewcommand{\dim}{\operatorname{dim}}

\newcommand{\ra}[1]{\renewcommand{\arraystretch}{#1}}

\newcommand{\De}{\mathcal{D}}

\newcommand{\Sym}{\operatorname{Sym}}

\newcommand{\supp}{\operatorname{supp}}

\newcommand{\Sp}{\operatorname{Sp}}

\newcommand{\R}{\mathbf{R}}
\newcommand{\Z}{\mathbf{Z}}

\renewcommand{\Gamma}{\varGamma}
\renewcommand{\epsilon}{\varepsilon}

\renewcommand{\hat}{\widehat}
\renewcommand{\leq}{\leqslant}
\renewcommand{\geq}{\geqslant}

\renewcommand{\S}{\mathbf{S} }
\newcommand{\B}{\mathcal{B} }

\renewcommand{\B}{\mathcal{B}}

\newcommand{\N}{\mathcal{N}}

\newcommand{\mF}{\mathbb{F}}

\newcommand{\G}{\mathcal{G}}

\newcommand{\U}{\mathcal{U}}
\renewcommand{\P}{\mathcal{P}}

\newcommand{\E}{\mathcal{E}}

\newcommand{\fix}{\operatorname{fix}}

\newcommand{\C}{\mathcal{C}}

\DeclareMathOperator{\Inc}{Inc} 

\begin{document}


\title[Equiangular lines, Incoherent Sets and Quasi-symmetric designs]{Equiangular lines, Incoherent Sets and Quasi-symmetric designs}


\author{Neil I. Gillespie}
\address{Heilbronn Institute for Mathematical Research, Department of Mathematics, University of Bristol, U.K.}
\email{neil.gillespie@bristol.ac.uk}



\begin{abstract} The absolute upper bound on the number of equiangular lines that can be found in $\R^d$ is
$d(d+1)/2$. Examples of sets of lines that saturate this bound are only known to exist in dimensions $d=2,3,7$ or $23$.
By considering the additional property of \emph{incoherence},  we prove that there exists a set of equiangular lines that saturates the 
absolute bound and the incoherence bound if and only if $d=2,3,7$ or $23$.
This allows us classify all \emph{tight spherical $5$-designs} $X$ in $\S^{d-1}$, the unit sphere, with the property
that there exists a set of $d$ points in $X$ whose pairwise inner products are positive. 

For a given angle $\kappa$, there exists a relative upper bound on the number of equiangular lines in $\R^d$ with common angle $\kappa$. 
We prove that classifying sets of lines that saturate this bound along with the incoherence bound is equivalent to classifying 
certain \emph{quasi-symmetric} designs, which are combinatorial designs with two block intersection numbers. 
Given a further natural assumption, we classify the known sets of lines that saturate these two bounds. This family comprises of the lines mentioned above
and the maximal set of $16$ equiangular lines found in $\R^6$.  There are infinitely many known sets of lines that saturate the relative bound,
so this result is surprising. To shed some light on this, we identify the $E_8$ lattice with the projection onto an $8$-dimensional subspace of a sublattice of the 
Leech lattice defined by $276$ equiangular lines in $\R^{23}$. This identification leads us to observe a 
correspondence between sets of equiangular lines in small dimensions and the exceptional curves of del Pezzo surfaces.
\end{abstract}


\subjclass[2010]{}

\maketitle

\section{Introduction}\label{sec:intro}
A set of lines in $\R^d$ is \emph{equiangular} if the angle between all pairs of lines is the same. In particular, given $n$ lines in $\R^d$, choose a unit vector
that spans each line. Then the set of lines is equiangular if the inner product between all unit vectors is $\pm \kappa$ where $\kappa$ is the cosine of the angle 
between them. A classical problem in Euclidean geometry is to determine the maximum
number $M(d)$ of equiangular lines in $\R^d$ for a given $d$. This problem seems to date back to Haantjes study in \cite{haantjes}, and was subsequently investigated
by van Lint and Seidel \cite{vanlintseid} and Lemmens and Seidel \cite{lemmens}. A connection with algebraic graph theory in \cite{vanlintseid} led
to significant work on equiangular lines and related problems, see for example \cite{bussseid, camcoh, mckayspence, seidelsurvey,taylor,tay2trans}.
Recent improvements have been made on various known upper and lower bounds for $M(d)$, which has led to renewed interest, see for example \cite{yu3,caen,yu1,greaves1, greaves2,yu4,yu2}; 
for the asymptotic case given a fixed angle, see \cite{keevash,zilin}.

Of the upper bounds mentioned above the best known is \emph{the absolute bound}. 
Gerzon proved that for any $d$, 
$$M(d)\leq \frac{d(d+1)}{2}.$$
However, this bound is only known to be a saturated when $d=2,3,7$ and $23$, and it is an open question
if it is saturated for any other values of $d$.

Associated with sets of equiangular lines is the notion of \emph{incoherence}. We say a set $\Gamma$ of equiangular lines is
an \emph{incoherent set} if either $|\Gamma|\leq 2$, or for any $3$-set of lines in $\Gamma$,
$$(\alpha_1,\alpha_2)(\alpha_1,\alpha_3)(\alpha_2,\alpha_3)>0$$
Here $\alpha_i$ is a unit vector that represents the corresponding line and $(\cdot,\cdot)$ denotes the standard Euclidean inner product. 
For any set $\Omega$ of equiangular lines in $\R^d$ it is known that
$$\Inc(\Omega)\leq d,$$
where $\Inc(\Omega)$ is the maximum size of any incoherent subset of lines contained in $\Omega$. 
We call this bound the \emph{incoherence bound}. The known sets of lines that saturate the absolute bound also saturate the incoherence bound.
Our first result proves that this is not a coincidence. 

\begin{Thm}\label{thm:main} Let $\Omega$ be a set of equiangular lines in $\R^d$ that saturates the absolute bound and the incoherence 
bound. Then $d=2,3,7$ or $23$, and $\Omega$ is isometric to
the corresponding set of lines given in Section \ref{sec:examples}. 
\end{Thm}

The problem of classifying sets of equiangular lines that saturate the absolute bound has also been studied from the point of view of \emph{spherical designs}.
Let $\S^{d-1}$ be the unit sphere in $\R^d$ for $d\geq 2$. A \emph{spherical $t$-design} is a finite set $X$ of points in $\S^{d-1}$ such that
$$\frac{1}{Vol(\S^{d-1})}\int_{{\bf{x}}\in \S^{d-1}}f({\bf{x}})d\sigma({\bf{x}})=\frac{1}{|X|}\sum_{{\bf{u}}\in X}f({\bf{x}})$$
for any polynomial $f\in\R[x_1,\ldots,x_d]$ of degree at most $t$. It is known that for any spherical $t$-design,
$$|X|\geq\binom{d+e-1}{d-1}+\binom{d+e-2}{d-1},\,\,\,|X|\geq 2\binom{d+e-1}{d-1} $$
for $t=2e$ and $t=2e+1$ respectively \cite{delgoeseid}. 
If either of these bounds is attained, then the spherical $t$-design is called \emph{tight}. 
For $d=2$, the tight spherical $t$-designs are the regular $(t+1)$-gons \cite[Example 5.13]{delgoeseid} and for $d\geq 3$, Bannai and Damerell proved that
a tight spherical $t$-design exists if and only if $t\in\{1,2,3,4,5,7,11\}$ \cite{bandd1,bandd2}. Furthermore, the tight spherical $t$-designs are classified for $t\in\{1,2,3,11\}$, 
and the existence of a tight spherical $4$-design is equivalent to the existence of a tight spherical $5$-design. Thus the open problem of classifying all tight spherical
$t$-deigns is reduced to classifying all tight spherical $5$ and $7$-designs. It is also known that a tight spherical $5$-design exists in $\S^{d-1}$ if and only if there exists 
a set of equiangular lines in $\R^d$ that saturate the absolute upper bound \cite[Theorem 5.12 and Example 8.3]{delgoeseid}. 
Therefore Theorem \ref{thm:main} implies the following classification of a family of tight spherical $5$-designs.



\begin{Thm}\label{thm:main2} Let $X$ be a tight spherical $5$-design in $\S^{d-1}$ that contains a set $\Gamma$ of $d$ points such that
the pairwise inner products of points in $\Gamma$ are positive. Then $d=2,3,7$ or $23$ and $X$ isometric to the intersection of $\S^{d-1}$ with set of equiangular lines in $\R^d$ given in section \ref{sec:examples}.
\end{Thm}

Returning to sets of equiangular lines in $\R^d$, it is known that for a given angle $\kappa$,
$$M(d,\kappa)\leq \frac{d-\kappa^2 d}{1-\kappa^2 d}$$
assuming that $\kappa^{2}d<1$, where $M(d,\kappa)$ is the maximum number of equiangular lines in $\R^d$ with angle $\kappa$. 
This is known as the \emph{relative upper bound}. As we have classified the sets of lines that saturate the absolute bound and the incoherence bound,
it is natural to try to do the same for sets of lines that saturate the relative bound. We prove the following structural result with regards to this question.

\begin{Thm}\label{thm:maineqquasi} Let $\Omega$ be a set of equiangular lines in $\R^d$ that saturates the relative bound, and let $\rho=\kappa^{-1}$. Then
$\Inc(\Omega)=d$ if and only if either
\begin{itemize}
\item[i)] $\Omega$ is the maximal set of $3$ equiangular lines in $\R^2$;
\item[ii)] $\Omega$ is the maximal set of $6$ equiangular lines in $\R^3$;
\item[iii)] $d=\rho(\rho-1)$, $|\Omega|=(\rho^2-1)(\rho-1)$, and there exists a 
quasi-symmetric 
$$2-(2i(2i+1)-1,(2i-1)(i+1),i(2i^2+i-2);i^2+i-1,i^2-1)$$
design, where $\rho=2i+1$ for an integer $i\geq 1$; or
\item[iv)] there exists a quasi-symmetric $2$-$(d,k,\lambda;s_1,s_2)$ design, where,
$k$ is a root of the quadratic
$$4x^2-4dx+(\rho-1)^2(d+\rho)$$
and 
$$\lambda=\frac{k(k-1)}{\rho^2-d},\,\,\,s_1=k-\frac{(\rho-1)^2}{4},\,\,\,s_2=k-\frac{(\rho^2-1)}{4}.$$
\end{itemize}
\end{Thm}

A \emph{quasi-symmetric design} is a combinatorial block design with exactly two block intersection numbers (see Section \ref{sec:prelims} for definitions). 
The only quasi-symmetric designs that satisfy the conditions of Theorem \ref{thm:maineqquasi} of which the author is aware are
the $S(2,2,n)$ Steiner systems for $n=5$ and $7$, and the $S(4,7,23)$ Steiner system.
The corresponding sets of equiangular lines are respectively the maximal sets of lines in dimensions $6,7$ and $23$. In fact, by considering an extra condition, we
can characterise the known sets of lines that saturate the relative and incoherence bounds, which we now explain.

Given any set $\Omega$ of equiangular lines in $\R^d$ and any pair of lines
$\alpha,\beta\in\Omega$, define
$$S_{\alpha\beta}=\{\gamma\in\Omega\,|\,(\alpha,\beta)(\alpha,\gamma)(\beta,\gamma)<0\}.$$
If $\Omega$ is saturates the relative bound, then it is known that the cardinality of $S_{\alpha\beta}$ is constant for all pair of lines. Moreover, 
if $\Gamma\subseteq\Omega$ is an incoherent subset of lines, then we prove that $$|S_{\alpha\beta}\cap S_{\alpha\gamma}|$$
is constant for all $3$-subsets $\{\alpha,\beta,\gamma\}$ of $\Gamma$. Additionally, for the set of lines in dimensions $6$, $7$ and $23$, we show
that
$$|S_{\alpha\beta}\cap S_{\eta\nu}|$$
is constant for all $4$-subsets $\{\alpha,\beta,\eta,\nu\}$ of any incoherent subset $\Gamma$ of $d$ lines. We prove that this extra property actually characterises the known sets of lines
that saturate the relative and incoherence bounds.

\begin{Thm}\label{thm:main4} Let $\Omega$ be a set of equiangular lines in $\R^d$ that saturates the relative bound
and the incoherence bound, and let $\Gamma\subseteq\Omega$ be an incoherent set of size $d$. If for $d\geq 4$,
$$|S_{\alpha\beta}\cap S_{\eta\nu}|$$
is constant for all $4$-subsets $\{\alpha,\beta,\eta,\nu\}$ of $\Gamma$, then
$d=2,3,6,7$ or $23$ and $\Omega$ is isometric to the corresponding set of equiangular lines given in Section \ref{sec:examples}.
\end{Thm}

One may have guessed that Theorem \ref{thm:main} would be true. The status quo since the 1970s has been that there are only four known sets of lines that saturate the absolute bound, 
and hence, one may have expected that there would be only finitely many sets of lines that saturate the absolute and incoherence bounds.
However, there are infinitely many known sets of lines that saturate the relative bound, but the above result is suggestive that 
only a finite number of them also saturate the incoherence bound. This is surprising (at least to the author). Moreover, one may ask why the maximal set of lines in $\R^6$ 
is the only other set of lines that appears in the above theorem.

To understand the known sets of lines that saturate the relative and incoherence bounds further, in Section \ref{sec:rootsE8} we look at the set of $276$ equiangular 
lines in $\R^{23}$ in more detail. In particular, we show how one can construct the roots of the $E_8$ lattice
from these lines. The consequences of this construction are interesting with respect to the Leech lattice. Using Conway's description of the Leech lattice \cite[Chapter 10]{conwaysloane}, let $\Lambda$
denote the Leech lattice and $\Lambda(n)$ denote the set of lattice vectors $v$ such that $v.v=16n$, so $\Lambda(2)$ is the set of minimal norm vectors of $\Lambda$.
By considering the stabiliser of a vector $v\in\Lambda(3)$, one can show that there exist $276$ pairs of antipodal lattice vectors in $\Lambda(5)$
that can be identified with the $276$ equiangular lines in $\R^{23}$, which we describe in Section \ref{sec:pfmain5}.
Therefore, the above construction of the roots of $E_8$ allows us to state the next result, which may be of independent interest.

\begin{Thm}\label{thm:main5} Let $\Lambda$ be the Leech lattice in $\R^{24}$ and $\Lambda_{276}$ be the sublattice of $\Lambda$ generated by $276$ antipodal lattice vector 
pairs of $\Lambda$ mentioned above (see Section \ref{sec:pfmain5}). Then there exists an $8$-dimensional
subspace $W$ of $\R^{24}$ such that the projection of $\Lambda_{276}$ onto $W$ can be identified with the $E_8$ lattice.
\end{Thm} 

This result allows us to identify various maximal sets of equiangular lines in lower dimensions with subsets of the $276$ lines, including the maximal set
of lines in $\R^6$. We also show a correspondence between
these sets of lines in lower dimensions and the exceptional curves of del Pezzo surfaces that is suggestive as to why the maximal set of lines in $\R^6$ also saturates 
the incoherence bound.

The layout of the rest of the paper is as follows. In Section \ref{sec:prelims} we introduce some standard results about $t$-designs, two-graphs and equiangular lines, 
and in Section \ref{sec:examples} we describe the known examples of sets of equiangular lines that saturate relative and incoherence bounds.
In Section \ref{sec:maxinc} we prove various results about sets of equiangular lines that saturate
the relative bound, including results on maximal incoherent subsets of such sets. Of particular importance is Theorem \ref{thm:eqquasi}, which proves that, in general, a quasi-symmetric design structure
must exist on any maximal incoherent $d$-subset of a set of equiangular lines that saturates the relative and incoherence bounds.
Using these results we are able to prove Theorem \ref{thm:main} and 
Theorem \ref{thm:main2} in Section \ref{sec:proofmain}. In Section \ref{sec:blocksets} we introduce \emph{block sets}, which are uniform $k$-hypergraphs with only two block intersection numbers. 
We show how one can construct sets of equiangular lines from block sets, and in particular, 
sets of lines that saturate the incoherence bound. Next, in Section \ref{sec:thm3}, we prove an analogous result to Theorem \ref{thm:eqquasi} which deals
with certain specific cases. Using these results we are then able to prove Theorem \ref{thm:maineqquasi}.
In Section \ref{sec:main4} we show that the known examples of sets of lines that saturate the relative and incoherence bounds satisfy
the conditions of Theorem \ref{thm:main4}, and subsequently we prove that theorem.
We also make two conjectures on these sets of lines. We prove Theorem \ref{thm:main5} in Section \ref{sec:rootsE8}.
In the final section we examine quasi-symmetric designs whose parameters satisfy the conditions of
Theorem \ref{thm:maineqquasi}, providing tables of some parameters, and we discuss certain non-existence results.

\section{Preliminaries}\label{sec:prelims}

\subsection{$t$-designs} A $t$-$(n,k,\lambda)$ design (or simply $t$-design) is a pair $\De=(\P,\B)$ where $\P$ is a point set of size $n$ and $\B$ is a set of $k$-sets (or blocks) of $\P$ with the property
that every $t$-set of $\P$ is contained in exactly $\lambda$ elements of $\B$. Any $t$-$(n,k,\lambda)$ design is also a $j$-$(n,k,\lambda_j)$ design for $0\leq j\leq t$, from which one deduces
that a $t$-$(n,k,\lambda)$ design exists only if for $0\leq j\leq t$
$$\lambda_j\binom{k-j}{t-j}=\lambda\binom{n-j}{t-j}.$$
By convention, $\lambda_0$, the total number of blocks, is denoted by $b$, and $\lambda_1$, the number of blocks that contain a given point, is denoted by $r$. 

Given a $t$-$(n,k,\lambda)$ design $\De=(\P,\B)$ and a point $p\in\P$, let $$\B_p=\{B\backslash\{p\}\,:\,B\in\B,p\in B\}$$ and
$$\B^p=\{B\,:\,B\in\B,p\notin\B\}.$$ Then the \emph{derived design $\De_p=(\P\backslash\{p\},\B_p)$} is a $(t-1)$-$(n-1,k-1,\lambda)$ design,
and the \emph{residual design $\De^p=(\P\backslash\{p\},\B^p)$} is a $(t-1)-(n-1,k,\lambda_{t-1}-\lambda_t)$ design.

Theorem \ref{thm:maineqquasi} relates sets of equiangular lines that saturate the relative bound and the incoherence bound to quasi-symmetric designs. 
A \emph{quasi-symmetric} design is a $t$-$(n,k,\lambda)$ design with exactly
two block intersection numbers, that is there exist integers $s_1,s_2$ such that all pairs of distinct blocks have exactly $s_1$ or $s_2$ elements in common. We refer to such a design as a
$t$-$(n,k,\lambda;s_1,s_2)$ design. These designs have been studied 
extensively due to their connections with strongly regular graphs.

\begin{Thm}\label{thm:quasisrg}\cite[Theorem 3.8]{shriksane} 
Let $(\P,\B)$ be a quasi-symmetric $2$-$(n,k,\lambda;s_1,s_2)$ design and form the \emph{block graph $\G$ of $(\P,\B)$} by taking as vertices
the blocks in $\B$, where two blocks are adjacent if and only if they have $s_1$ elements in common. Then if $\G$ is connected it is a $(b,t,p,q)$ strongly regular graph
with eigenvalues
$$\theta_0=\frac{k(r-1)-(b-1)s_2}{s_1-s_2},\,\,\,\,\,\theta_1=\frac{(r-\lambda)-(k-s_2)}{s_1-s_2},\,\,\,\,\,\theta_2=-\frac{k-s_2}{s_1-s_2}$$
where
$$t=\theta_0,\,\,\,\,\,\,\,p=\theta_0+\theta_1+\theta_2+\theta_1\theta_2,\,\,\,\,\,q=\theta_0+\theta_1\theta_2.$$
\end{Thm}

We note if $s_1<k$, then the block graph $\G$ of a quasi-symmetric design is connected \cite[Theorem 0]{neum}.
The following result gives some necessary conditions for the existence of quasi-symmetric designs.

\begin{Thm}\label{thm:calder}\cite{calderbank}
Let $(\P,\B)$ be a $2$-$(n,k,\lambda)$ design such that $|B_1\cap B_2|=k-x$ or $k-y$ for all pairs of distinct blocks $B_1,B_2\in\B$. Then
\begin{equation}\label{eq:calder}f(n,k,x,y)=(n-1)(n-2)xy-k(n-k)(n-2)(x +y) +k(n- k)(k(n - k)- 1)\geq 0\end{equation}
Moreover, $f(n,k,x,y)=0$ if and only if $(\P,\B)$ is a $3$-design.
\end{Thm}

We refer the reader to \cite{camvan} and \cite{shriksane}, and references therein, for more details on $t$-designs, and more specifically, quasi-symmetric designs.

\subsection{Two-graphs, Switching classes of Graphs, and Equiangular lines}
A \emph{two-graph} on a finite point set $\P$ of size $n$ is a pair $(\P,\B)$ where $\B$ is a set of $3$-sets of $\P$ such that every $4$-set of $\P$ contains an even number of 
elements of $\B$ as subsets. 
A two-graph $(\P,\B)$ is \emph{regular} if it is also a $2$-$(n,3,a)$ design for some integer $a$.

Given a two-graph $(\P,\B)$, we say that a $k$-set $B$ is a \emph{coherent} $k$-set if $|B|\geq 3$ and every
$3$-set of $B$ is an element of $\B$, and we say $B$ is an \emph{incoherent} $k$-set if $|B|\leq 2$ or $|B|\geq 3$ and every 
$3$-set of $B$ is not an element of $\B$. If $(\P,\B)$ is a regular two-graph on $n$ points, and therefore is a $2$-$(n,3,a)$ design, then it is known
that every coherent $3$-set is contained in $b$ coherent $4$-sets, where $n=|\P|=3a-2b$ \cite[Proposition 3.1]{taylor}. If we want to make reference to all
three of these parameters, we say $(\P,\B)$ is a regular two-graph with parameters $(n,a,b)$. For any regular two-graph $(\P,\B)$ with parameters $(n,a,b)$,
the complementary two-graph $(\P,\B^*)$ is a regular two-graph with parameters $(n,a^*,b^*)$, where $\B^*=\P^{\{3\}}\backslash\B$ and $a^*=n-a-2$, $b^*=n/2-b-3$. 

For a point set $\P$, the \emph{complete} two-graph is $(\P,\P^{\{3\}})$ and the \emph{null} two-graph is $(\P,\emptyset)$. Other examples of two-graphs come from
graphs. Namely, let $\G=(V\G,E\G)$ be an undirected graph with vertex set $V\G$ and edge set $E\G$. 
Then the two-graph associated with $\G$ is $(V\G,\B_G)$ where $\B_G$ consists of the $3$-sets of $V\G$ with the property 
that the induced subgraph on the three vertices has an odd number of edges.

A two-graph on $n$ points is equivalent to a \emph{switching class of graphs on $n$ vertices}. 
Given a graph $\G$ and a subset $X$ of vertices of $V\G$, the operation of switching with respect to $X$ interchanges edges and non-edges between
$X$ and its complement $V\G\backslash X$, and leaves all other edges and non-edges alone.
Switching forms an equivalence relation on the set of all graphs on $n$ vertices, and switching equivalent graphs determine the same two-graph.

We can also construct two-graphs from sets of equiangular lines in $\R^d$. 
Let $\Omega$ be a set of equiangular lines in $\R^d$ and let $$\C=\{\{\alpha,\beta,\gamma\}\in\Omega^{\{3\}}\,:\,(\alpha,\beta)(\beta,\gamma)(\gamma,\alpha)<0\}.$$
(Here, and throughout, if the context is clear, we let $\alpha$ denote both a line in $\Omega$ and a unit vector representing that line.)
Then $(\Omega,\C)$ forms a two-graph. This can be seen by first choosing a set $\U(\Omega)$ of unit vectors, 
each one representing a line in $\Omega$. Now define the graph on $\U(\Omega)$ where two vectors $\alpha_1,\alpha_2$ are adjacent if and only if $(\alpha_1,\alpha_2)<0$.
By identifying the unit vectors with the lines that they represent, the corresponding two-graph of this graph is isomorphic to $(\Omega,\C)$. 
We note that this two-graph is independent of the choice of unit vectors, that is, by choosing another set of unit vectors to represent the lines, we define
a graph that is switching equivalent to the original one.

One can also construct a set of equiangular lines in Euclidean space from any graph. Indeed, two-graphs on $n$ points, switching classes of graphs on $n$ vertices, 
and linearly dependent $n$-sets of equiangular lines in Euclidean space (up to isometry) are equivalent objects. We refer the reader to \cite{seidelsurvey} for more details.
Given this equivalence, we shall in the sequel use the description (equiangular lines, two-graphs or switching classes) that best suits the circumstances. 

The next three results on the various upper bounds and angles for sets of equiangular lines are necessary in the sequel.

\begin{Thm}\label{thm:absbd}\cite[Theorem 3.5 - Due to M. Gerzon]{lemmens} Let $\Omega$ be a set of equiangular lines in $\R^d$. Then $$|\Omega|\leq \frac{d(d+1)}{2}.$$
If equality holds then $d+2=4,5$, or the square of an odd integer. Moreover the corresponding angle is given by $\cos\theta=\frac{1}{\sqrt{d+2}}$.
\end{Thm}

\begin{Thm}\label{thm:relbd}\cite[Theorem 3.6]{lemmens} Let $\Omega$ be a set of equiangular lines in $\R^d$ with common angle $\kappa$ and suppose that $\kappa^2d<1$. Then
$$|\Omega|\leq \frac{d-\kappa^2d}{1-\kappa^2d}.$$ Moreover, assuming $\kappa^2d<1$, this bound is saturated if and only if  the two-graph $(\Omega,\C)$ is regular.
\end{Thm}

\begin{Thm}\label{thm:rhoodd}\cite[Theorem 3.4 - Due to P. Neumann]{lemmens} Let $\Omega$ be a set of equiangular lines in $\R^d$ with common angle $\kappa$, 
and let $\rho=\kappa^{-1}$. If $|\Omega|>2d$, then $\rho$ is an odd integer.
\end{Thm}

\section{Examples}\label{sec:examples}
In this section we present the known examples of sets of equiangular lines that saturate the relative and incoherence bounds, which are
the maximal sets of equiangular lines in dimensions $d=2,3,6,7$ and $23$. 
The maximum number of equiangular lines in $\R^2$ is $3$. These can be constructed by taking 
a regular hexagon centred at the origin and drawing the three lines that connect diagonally opposite vertices.
Any incoherent set of equiangular lines in $\R^d$ forms a linearly independent set, and so $3$ equiangular lines in $\R^2$ cannot form an incoherent set. 
Thus, by definition, the maximum size of any incoherent set in this case is $d=2$.

In $\R^3$ six equiangular lines can be constructed by drawing lines through opposite vertices of an icosahedron. 
The icosahedron has $12$ vertices, $30$ edges and $20$ triangular faces. The set of $3$ lines formed from the vertices of one of the triangular faces of the icosahedron is an incoherent set.

For dimensions $d=6,7$ and $23$, we construct $M(d)-d$ lines of the form $v(B)$ for some block in
a quasi-symmetric $2$-design, and $d$ lines of the form $v(i)$ for $i=1,\ldots,d$ using the parameters of this design.

Let $(\P,\B)$ be a quasi-symmetric $2$-$(d,k,\lambda ;s_1,s_2)$ design, and $$\Delta_1=k^2-\frac{d(s_1+s_2)}{2}$$
For each $B\in\B$ let $v(B)\in\R^d$ whose $j$th entry is $$v(B)_j=\begin{cases}{d-k}+\sqrt{\Delta_1}&\textnormal{ if $j\in B$ }\\
{-k}+\sqrt{\Delta_1}&\textnormal{ if $j\in\P\backslash B$ }
\end{cases}$$
Now let
$$\Delta_2=(k-s_1)+\frac{d(s_1-s_2)}{2}$$
and for $i\in\P$ let $v(i)\in\R^d$ with
$$v(i)_j=\begin{cases}(s_1-s_2)(d-1)-\sqrt{\Delta_2}&\textnormal{ if $j=i$ }\\
-(s_1-s_2)-\sqrt{\Delta_2}&\textnormal{ if $j\neq i$ }
\end{cases}$$
For $d=6$, let $(\P,\B)$ be the unique $2$-$(6,3,2;2,1)$ quasi-symmetric design; for $d=7$, let $(\P,\B)$ be the 
$2$-$(7,2,1;1,0)$ quasi-symmetric design, that is, $\B$ is the set of all $2$-sets of $\P=\{1,\ldots,7\}$; for $d=23$, let $(\P,\B)$ be the $S(4,7,23)$ Steiner system, 
which is a $4$-$(23,7,1;3,1)$ quasi-symmetric design. With these specified quasi-symmetric designs, we find that
$$\Omega=\{v(B)\,|\,B\in\B\}\cup \{v(i)\,|\,i\in\P\}$$
forms a set of $M(d)$ vectors in $\R^d$ that span equiangular lines in the respective dimension, and moreover that
$$\Gamma=\{v(i)\,|\,i\in\P\}$$
is an incoherent subset of size $d$.

\begin{Rem}\label{rem:uniiso}
Recall that sets of equiangular lines that saturate the absolute bound are equivalent to tight spherical $5$-designs.
In particular, the only known examples of tight spherical $5$-designs can be found
by taking the intersection points of the unit sphere with the known sets of equiangular lines that saturate the absolute upper bound in dimensions $d=2,3,7$ or $23$. It is further known
that the tight spherical $5$-designs in these dimension $2$, $3$ and $7$ are unique up to isometry 
(for $d=2,3$, see \cite[Examples 5.13 \& 5.16]{delgoeseid}; for $d=7$, see \cite[Theorem 11]{bannaisloane}).
For $d=23$, the only claim of this fact that we are aware of (see \cite{bannaivenkov}) relies on the uniqueness the regular two-graph on $276$ points. 
In \cite{seidelsurvey} Seidel proved a one to one correspondence between dependent sets of $n$ equiangular lines (up to isometry) and 
two-graphs on $n$ points. Hence the uniqueness of the regular two-graph on $276$ points \cite{276lines} implies that the corresponding set of $276$ equiangular lines in $\R^{23}$ is 
unique up to isometry. Indeed, Taylor proved that the regular two-graphs on $n=6,16,28$ points are also unique \cite{taylor}, proving that the maximal sets of equiangular lines in dimensions $d=3,6,7$ are unique up to isometry. 
\end{Rem}

\section{Regular two-graphs, Maximal Incoherent Subsets and $2$-designs}\label{sec:maxinc}

In this section we construct various $2$-designs from regular two-graphs and maximal incoherent subsets of equiangular lines.
Of particular importance is Lemma \ref{lem:setsum}, which gives some necessary conditions for maximal incoherent subsets of regular two-graphs.
We use this to prove, in Theorem \ref{thm:eqquasi}, that for certain sets of equiangular lines which saturate the relative and incoherence bounds, a quasi-symmetric
design structure must be present in any incoherent subset of $d$ lines. 

\subsection{Regular two-graphs and $2$-designs}
Let $(\Omega,\C)$ be a regular two-graph. We denote by $\C_0$ the set of coherent $4$-sets of $(\Omega,\C)$, 
$\C_2$ the set of incoherent $4$-sets of $(\Omega,\C)$, and $\C_1=\Omega^{\{4\}}\backslash(\C_0\cup \C_2)$. The definition
of a two-graph implies that each element $B\in\C_1$ has exactly two elements of $\C$ as subsets, that is, $|B^{\{3\}}\cap\C|=2$.

\begin{Prop}\label{prop:32des} Let $(\Omega,\C)$ be a regular two graph with parameters $(n,a,b)$, and let $\C_i$ for $i=0,1,2$ be defined as above.  Then
\begin{itemize}
\item[i)] $(\Omega,\C_0)$ forms a $2-(n,4,ab/2)$ design;
\item[ii)] $(\Omega,\C_1)$ is a $2-(n,4,3aa^*/2)$ design; and 
\item[iii)] $(\Omega,\C_2)$ forms a $2-(n,4,a^*b^*/2)$ design,
\end{itemize}
where $a^*=n-a-2$ and $b^*=n/2-b-3$.
\end{Prop}

\begin{proof}
Let $\alpha,\beta\in\Omega$. Then there exist exactly $a$ coherent triples that contain $\alpha,\beta$, and for each of these coherent triples, there exist
exactly $b$ coherent $4$-sets that contains the coherent triple. 
Thus we have just counted $ab$ coherent $4$-sets that contain $\alpha$ and $\beta$. However, we have double counted, because for every coherent $4$-set $B$ that contains
$\alpha$ and $\beta$, there are exactly two coherent triples in $B$ that contain $\alpha$ and $\beta$. 
Hence there are a total of $ab/2$ coherent $4$-sets that contain $\alpha,\beta$, proving i). 
As the complement of a regular two-graph is also a regular two-graph with parameters $(n,a^*,b^*)$, where $a^*=n-a-2$
and $b^*=n/2-b-3$, iii) now follows. Finally, ii) holds as the $(\Omega,\C_1)$ is necessarily a $2-(n,4,\lambda)$ design where
$$\lambda=\frac{(n-2)(n-3)}{2}-\frac{ab}{2}-\frac{a^*b^*}{2}=\frac{3aa^*}{2}.$$
\end{proof}

Given a two-graph $(\Omega,\C)$, for $\alpha,\beta\in\Omega$ we define 
$$S_{\alpha\beta}=\{\gamma\in\Omega\,|\,\{\alpha,\beta,\gamma\}\in\C\}.$$
Obviously if $(\Omega,\C)$ is a regular two-graph with parameters $(n,a,b)$, each $S_{\alpha\beta}$ contains exactly $a$ elements.
In fact, a design structure is present in this case.

\begin{Thm}\label{thm:desint1}
Let $(\Omega,\C)$ be a regular two-graph with parameters $(n,a,b)$. Then $(\Omega,\B)$ is a $2$-$(n,a,a(a-1)/2)$ design, where
$$\B=\{S_{\alpha\beta}\,|\,\alpha,\beta\in\Omega\}.$$ Moreover, \begin{equation}\label{eq:desint}|S_{\alpha\beta}\cap S_{\alpha\gamma}|=
\begin{cases} b&\textnormal{ if $\{\alpha,\beta,\gamma\}$ is a coherent $3$-set}\\
a/2&\textnormal{ if $\{\alpha,\beta,\gamma\}$ is an incoherent $3$-set}
\end{cases}\end{equation}
\end{Thm}

\begin{proof}
Let $\alpha,\beta\in\Omega$, and note that $\alpha\notin S_{\alpha\gamma'}$ for
any $\gamma'\in\Omega\backslash\{\alpha\}$, and similarly for $\beta$. Thus if $\alpha,\beta\in S_{\gamma\delta}$ then $\{\alpha,\beta,\gamma,\delta\}$ is a $4$-set. 
As $\{\alpha,\gamma,\delta\}\in\C$ and $\{\beta,\gamma,\delta\}\in\C$, we conclude that either 
$$\gamma,\delta\in S_{\alpha\beta}\,\,\,\textnormal{ or }\,\,\,\gamma,\delta\notin S_{\alpha\beta}.$$

If $\gamma,\delta\in S_{\alpha\beta}$ then $\{\alpha,\beta,\gamma,\delta\}$ is a coherent $4$-set. Moreover, note that $\alpha,\beta\in S_{\gamma'\delta'}$
for any coherent $4$-set $\{\alpha,\beta,\gamma',\delta'\}$, and by Proposition \ref{prop:32des}, there are $ab/2$ coherent $4$-sets that contain $\alpha,\beta$.

If $\gamma,\delta\notin S_{\alpha\beta}$ then $\{\alpha,\beta,\gamma\}$ and $\{\alpha,\beta,\delta\}$ are incoherent $3$-sets and $\{\alpha,\beta,\gamma,\delta\}$ is not an
incoherent $4$-set. As $\alpha,\beta$ are contained in $a^*$ incoherent $3$-sets, there are a total of $a^*(a^*-1)/2$ $4$-sets $\{\alpha,\beta,\gamma',\delta'\}$ that contain $\alpha,\beta$ such that
$\gamma',\delta'\notin S_{\alpha\beta}$. Of these, $a^*b^*/2$ are incoherent $4$-sets by Proposition \ref{prop:32des}. Thus there are
$$a^*(a^*-1)/2-a^*b^*/2=aa^*/4$$
$4$-sets with the desired property. Hence there are exactly $$ab/2+aa^*/4=a(a-1)/2$$ 
sets $S_{\gamma\delta}$ that contain $\alpha,\beta$, proving the first part of the result. 

Suppose that $\{\alpha,\beta,\gamma\}$ is a coherent $3$-set. Then the defining two-graph properties imply that
$$\delta\in S_{\alpha\beta}\cap S_{\alpha\gamma} \iff \{\alpha,\beta,\gamma,\delta\}\textnormal{ is a coherent $4$-set}.$$
As $(\Omega,\C)$ is a regular two-graph, there are exactly $b$ elements $\delta\in\Omega$ such that $\{\alpha,\beta,\gamma,\delta\}$ is a coherent $4$-set, proving the first intersection result.

Now suppose that $\{\alpha,\beta,\gamma\}$ is an incoherent $3$-set. Then for $\delta\in\Omega\backslash\{\alpha,\beta,\gamma\}$,
$$\delta\notin S_{\beta\gamma} \iff \delta\in S_{\alpha\beta}\cap S_{\alpha\gamma}\textnormal{ or } \delta\notin S_{\alpha\beta}\cup S_{\alpha\gamma}.$$
By considering the complementary two-graph, we note that there exist $b^*$ elements $\delta\in\Omega\backslash\{\alpha,\beta,\gamma\}$ such that 
$\delta\notin S_{\alpha\beta}\cup S_{\alpha\gamma}$. Hence
$$|S_{\alpha\beta}\cap S_{\alpha\gamma}|=n-3-a-b^*=a/2.$$
\end{proof}

\subsection{Previous results on Maximal Incoherent Subsets}\label{sec:taylor}
Before we prove our results on maximal incoherent subsets of equiangular lines, we mention some previous work on this topic. 
In \cite{taylor} Taylor investigated maximal coherent sets of (regular) two-graphs. His results are relevant for our purposes because any maximal coherent set of
a two-graph is a maximal incoherent set of the complementary two-graph. We present his results (in the language of equiangular lines) that we use in the sequel. First we state a well known
fact about an incoherent set of equiangular lines in $\R^d$.  

\begin{Lem}\label{lem:incupbd} Let $\Gamma$ be an incoherent set of equiangular lines in $\R^d$. Then $\Gamma$ forms a linearly independent set. In particular, $|\Gamma|\leq d$.
\end{Lem}

Let $\Omega$ be a set of equiangular lines and $\Gamma\subseteq\Omega$ be a maximal incoherent subset. Then by definition, for each
$\gamma\in\Omega\backslash\Gamma$ there exist $\alpha_1,\alpha_2\in\Gamma$ such that $\{\gamma,\alpha_1,\alpha_2\}$ is a coherent set. 
Taylor \cite{taylor} used this fact to define the following two sets:
\begin{equation}\label{eq:defsets}\Gamma_i(\gamma)=\{\delta\in\Gamma\,|\,\{\gamma,\alpha_i,\delta\}\textnormal{ is a coherent $3$-set}\}\end{equation}
for $i=1,2$. One can deduce from the defining properties of two-graphs that these sets are independent of the choice of $\alpha_1,\alpha_2$, and that they 
partition $\Gamma$. We label the sets so that 
$$|\Gamma_1(\gamma)|\leq|\Gamma_2(\gamma)|$$
The following observation is relevant in the sequel.

\begin{Lem}\label{lem:simple}
Consider the two-graph $(\Omega,\C)$ and let $\Gamma$ be a maximal incoherent subset of $\Omega$, $\gamma\in\Omega\backslash\Gamma$ and $\alpha,\beta\in\Gamma$. Then
$$\{\gamma,\alpha,\beta\}\notin\C\iff \alpha,\beta\in\Gamma_i(\gamma)$$
for $i=1$ or $2$.
\end{Lem}

\begin{proof} This result holds as $\alpha,\beta\in\Gamma_i(\gamma)$ for $i=1$ or $2$
if and only if there exists $\delta\in\Gamma\backslash\Gamma_i(\gamma)$ such
that $\{\gamma,\delta,\beta\}\in\C$ and $\{\gamma,\delta,\alpha\}\in\C$ if and only if $\{\gamma,\alpha,\beta\}\notin\C$ (by the two-graph property).
\end{proof}

By supposing that there exists a maximal incoherent subset that saturates the upper bound in Lemma \ref{lem:incupbd}, Taylor proves the following results.
(In the statements of the theorems by Taylor, he assumes that $(\Omega,\C)$ is a regular two-graph. 
However the regularity property is not used in his proofs, so we present the more general statements.)

\begin{Prop}\label{prop:tayeq}\cite[Theorem 5.4]{taylor} Let $\Omega$ be a set of equiangular lines in $\R^d$ with common angle $\kappa$, 
and let $\Gamma\subseteq\Omega$ be an incoherent set of $d$ lines. Then for each 
$\gamma\in\Omega\backslash\Gamma$,
$|\Gamma_1(\gamma)|$ and $|\Gamma_2(\gamma)|$ are the solutions to the quadratic equation
\begin{equation}\label{eq:taypoly}4x^2-4dx+(\rho-1)^2(d+\rho)\end{equation}
where $\rho=\kappa^{-1}$. In particular, 
\begin{equation}\label{eq:tayeq}|\Gamma_j(\gamma)|=\frac{d+(-1)^j\sqrt{d^2-(\rho-1)^2(d+\rho)}}{2}.\end{equation}
for $j=1,2$.
\end{Prop}

\begin{Rem}\label{rem:coeff} 
If $\Omega$ is a set of equiangular lines $\R^d$ (with common angle $\kappa=\rho^{-1}$) that contains a incoherent set $\Gamma$ of $d$ lines, then by Lemma \ref{lem:incupbd}, 
any set of unit vectors representing the lines in $\Gamma$ forms a basis for $\R^d$. Moreover, these unit vectors can be chosen so that
their pairwise inner product is positive. Given such a basis $\mathfrak{B}=\{\alpha_i\}_{i=1}^d$, Taylor showed that for $\gamma\in\Omega\backslash\Gamma$, the unit vector 
$$\frac{2d-2k+\rho-1}{(\rho-1)(d+\rho-1)}\sum_{\alpha\in\Gamma_1(\gamma)}\alpha-\frac{(2k+\rho-1)}{(\rho-1)(d+\rho-1)}\sum_{\alpha\in\Gamma_2(\gamma)}\alpha$$
spans $\gamma$. Here $k=|\Gamma_1(\gamma)|$, and note, we are identifying the lines in $\Gamma_i(\gamma)$ ($i=1,2$) with the set of unit vectors in $\mathfrak{B}$ that span them.
\end{Rem}

\begin{Prop}\label{prop:tayint}\cite[Theorem 5.6]{taylor} Let $\Omega$ be a set of equiangular lines in $\R^d$ with common angle $\kappa$ 
and $\Gamma\subseteq\Omega$ be an incoherent subset of $d$ lines. Then for distinct elements
$\gamma,\delta\in\Omega\backslash\Gamma$,
$$|\Gamma_1(\gamma)\cap\Gamma_1(\delta)|=|\Gamma_1(\gamma)|-\Delta$$
where $\Delta=(\rho-1)^2/4$ or $(\rho^2-1)/4$ and $\rho=\kappa^{-1}$.
\end{Prop}

\subsection{Maximal Incoherent Subsets and $2$-designs}
Let $\Omega$ be a set of equiangular lines in $\R^d$ and let $\Gamma\subseteq\Omega$ be a maximal incoherent subset of lines. 
In the case that $(\Omega,\C)$ forms a regular two-graph with parameters $(n,a,b)$, Taylor~\cite{taylor} proved the following identity:
\begin{equation}\label{eq:tayloreq}\sum_{\gamma\in\Omega\backslash{\Gamma}}|\Gamma_1(\gamma)||\Gamma_2(\gamma)|=\frac{a|\Gamma|(|\Gamma|-1)}{2}.\end{equation}
We now prove similar identities to the one above, which then allow us to relate $2$-designs to maximal incoherent subsets. First we introduce the following sets.

Let $\gamma\in\Omega\backslash{\Gamma}$ and $\alpha\in\Gamma$. We define
$$\Gamma(\gamma,\alpha,\in),\,\,\Gamma(\gamma,\alpha,\notin)$$
to be the set in $\{\Gamma_1(\gamma),\Gamma_2(\gamma)\}$ that does, respectively does not, contain $\alpha$. 

\begin{Lem}\label{lem:setsum} Let $\Gamma$ be a maximal incoherent subset in a regular two-graph $(\Omega,\C)$ with parameters $(n,a,b)$, and suppose $|\Gamma|=g$.
Then for a fixed $\alpha\in\Gamma$,
\begin{align}
\label{eq:setsum1}\sum_{\gamma\in\Omega\backslash{\Gamma}}|\Gamma(\gamma,\alpha,\in)|&=(n-g-a)g+a\\
\label{eq:setsum2}\sum_{\gamma\in\Omega\backslash{\Gamma}}|\Gamma(\gamma,\alpha\notin)|&=a(g-1)\\
\label{eq:setsum}\sum_{\gamma\in\Omega\backslash{\Gamma}}|\Gamma(\gamma,\alpha,\in)|^2&=(n-g-\frac{3a}{2})g^2+\frac{3a}{2}g\\
\label{eq:setsum3}\sum_{\gamma\in\Omega\backslash{\Gamma}}|\Gamma(\gamma,\alpha,\notin)|^2&=\frac{ag(g-1)}{2}.
\end{align}
Moreover, for $\beta\in\Gamma\backslash\{\alpha\}$,
\begin{align}
\label{eq:setsum4}\sum_{\gamma\in S_{\alpha\beta}}|\Gamma(\gamma,\alpha,\in)|&=\sum_{\gamma\in S_{\alpha\beta}}|\Gamma(\gamma,\alpha,\notin)|=\frac{ag}{2},\\
\label{eq:setsum5}\sum_{\gamma\in\Omega\backslash{(\Gamma\cup S_{\alpha\beta})}}|\Gamma(\gamma,\alpha,\in)|&=(n-g-\frac{3a}{2})g+a,\\
\label{eq:setsum6}\sum_{\gamma\in\Omega\backslash{(\Gamma\cup S_{\alpha\beta})}}|\Gamma(\gamma,\alpha,\notin)|&=\frac{a(g-2)}{2}.
\end{align}
\end{Lem}

\begin{proof} 
We first show that \eqref{eq:setsum2} holds. We claim that, given $\alpha\in\Gamma$, $$\bigcup_{\beta\in\Gamma\backslash\{\alpha\}}S_{\alpha\beta}=\Omega\backslash\Gamma.$$
Let $\gamma\in\Omega\backslash\Gamma$. Then there exists $\alpha_1,\alpha_2\in\Gamma$ such that $\{\alpha_1,\alpha_2,\gamma\}\in\C$. If $\alpha=\alpha_i$ for $i=1$ or $2$, then
$\gamma\in S_{\alpha\alpha_j}$, where $j\neq i$. Otherwise, consider the $4$-set $\{\alpha,\alpha_1,\alpha_2,\gamma\}$. 
Since $\Gamma$ is an incoherent set, it follows from the two-graph property that $\{\alpha,\alpha_i,\gamma\}\in\C$ for $i=1$ or $2$. In particular, $\gamma\in S_{\alpha\alpha_i}$.

By Lemma \ref{lem:simple}, for $\beta\in\Gamma\backslash\{\alpha\}$, $$\gamma\in S_{\alpha\beta}\iff\{\gamma,\beta,\alpha\}\in\C\iff\beta\in\Gamma(\gamma,\alpha,\notin).$$ Thus counting in two ways the pairs $$\{\{\gamma,\beta\}\in\Omega\backslash\Gamma\times\Gamma\backslash\{\alpha\}\,|\,\{\gamma,\beta,\alpha\}\in\C\}$$
we deduce that \eqref{eq:setsum2} holds. Since 
$$|\Gamma(\gamma,\alpha,\in)|+|\Gamma(\gamma,\alpha,\notin)|=g$$
it follows that \eqref{eq:setsum1} is equal to $(n-g)g-a(g-1)=(n-g-a)g+a$

We deduce from \eqref{eq:tayloreq} and \eqref{eq:setsum2} that
\begin{align*}ag(g-1)&=\sum g|\Gamma(\gamma,\alpha,\notin)|=\sum |\Gamma(\gamma,\alpha,\notin)|^2+\sum |\Gamma_1(\gamma)||\Gamma_2(\gamma)|\\
&=\sum |\Gamma(\gamma,\alpha,\notin)|^2+\frac{ag(g-1)}{2}.
\end{align*}
Solving the above equation gives \eqref{eq:setsum3}. An almost identical argument gives \eqref{eq:setsum}.

Now let $\beta\in\Gamma\backslash\{\alpha\}$. For all $\gamma\in S_{\alpha\beta}$, it follows from Lemma \ref{lem:simple} that
$$\Gamma(\gamma,\alpha,\notin)=\Gamma(\gamma,\beta,\in).$$
Moreover, for all $\gamma\in\Omega\backslash(\Gamma\cup S_{\alpha\beta})$
$$\Gamma(\gamma,\alpha,\in)=\Gamma(\gamma,\beta,\in).$$
Thus by \eqref{eq:setsum1},
$$\sum_{\gamma\in S_{\alpha\beta}}|\Gamma(\gamma,\alpha,\in)|=\sum_{\gamma\in S_{\alpha\beta}}|\Gamma(\gamma,\beta,\in)|=\sum_{\gamma\in S_{\alpha\beta}}|\Gamma(\gamma,\alpha,\notin)|.$$
Hence
$$ag=\sum_{\gamma\in S_{\alpha\beta}}|\Gamma(\gamma,\alpha,\in)|+|\Gamma(\gamma,\alpha,\notin)|=2\sum_{\gamma\in S_{\alpha\beta}}|\Gamma(\gamma,\alpha,\in)|$$
which proves \eqref{eq:setsum4}. Subtracting \eqref{eq:setsum4} from \eqref{eq:setsum1}, respectively from \eqref{eq:setsum2}, gives \eqref{eq:setsum5}, respectively \eqref{eq:setsum6}.
\end{proof}

\begin{Thm}\label{thm:linedes} Let $\Omega$ be a set of equiangular lines in $\R^d$ such that
$(\Omega,\C)$ is a regular two-graph with parameters $(n,a,b)$. Moreover, let $\Gamma\subseteq\Omega$ be a maximal incoherent subset such that $|\Gamma_1(\gamma)|=g_1$ for all $\gamma\in\Omega\backslash\Gamma$, and let $g_2=|\Gamma|-g_1$. Then 
if $g_1\neq g_2$,
$(\Gamma,\B_i)$ is a $2$-$(g,g_i,\lambda_i)$ design for $i=1,2$ where 
$$\B_i=\{\Gamma_i(\gamma)\,|\,\gamma\in\Omega\backslash\Gamma\}$$
and \begin{equation}\label{eq:lam}\lambda_i=\frac{a(g_i-1)}{2g_j},\,\,\,\,\,i\neq j.\end{equation}
Moreover, if $g_1=g_2$ then
$(\Gamma,\B)$ is a $2$-$(g,g/2,n-g-a)$ design where $\B=\B_1\cup\B_2$.
\end{Thm} 

\begin{proof}
Suppose that $|\Gamma|=g$ and let $\alpha,\beta\in\Gamma$. By Lemma \ref{lem:simple},
for all $\gamma\in S_{\alpha\beta}$, 
\begin{equation}\label{eq:ab1}\alpha\in\Gamma_i(\gamma),\,\,\beta\in\Gamma_j(\gamma)\end{equation}
where $i\neq j$, and for all $\gamma\in\Omega\backslash(\Gamma\cup S_{\alpha\beta})$,
\begin{equation}\label{eq:ab2}\alpha,\beta\in\Gamma_i(\gamma)\end{equation}
where $i=1$ or $2$. Let $k_i$ be the number of $\gamma\in\Omega\backslash(\Gamma\cup S_{\alpha\beta})$ such that
$\alpha,\beta\in\Gamma_i(\gamma)$ for $i=1,2$. Then 
$$k_1+k_2=n-a-g$$
and by \eqref{eq:setsum5}
$$k_1g_1+k_2g_2=(n-g-\frac{3a}{2})g+a.$$
Assuming $g_1\neq g_2$, these two linear equations can be solved to give 
$$k_1=\frac{(2(n-g)-3a)g+2a-2g_2(n-g-a)}{2(g_1-g_2)}=\frac{a(g_1-1)}{2g_2},$$
and
$$k_2=\frac{2g_1(n-g-a)-(2(n-g)-3a)g-2a}{2(g_1-g_2)}=\frac{a(g_2-1)}{2g_1}.$$
The identities on the right hand side can be determined using \eqref{eq:tayloreq} and the fact that $g=g_1+g_2$.
Thus every pair of elements in $\Gamma$ is contained in $k_i$ elements of $\B_i$ for $i=1,2$, proving the result.

If $g_1=g_2=g/2$, then it follows from \eqref{eq:ab1} and \eqref{eq:ab2} that $(\Gamma,\B)$ is a $2$-$(g,g/2,n-g-a)$ design.

\end{proof}

\begin{Thm}\label{thm:eqquasi} Let $\Omega$ be a set of equiangular lines in $\R^d$ such that $(\Omega,\C)$ is a regular two-graph
and $\Inc(\Omega)=d$. Let $\Gamma\subseteq\Omega$ be an incoherent subset of $d$ lines,
$$\B_1=\{\Gamma_1(\gamma)\,|\,\gamma\in\Omega\backslash\Gamma\}$$
and $\rho=\kappa^{-1}$. Suppose that $|\Gamma_1(\gamma)|<d/2$ for some $\gamma\in\Omega\backslash\Gamma$. Then
either 
\begin{itemize}
\item[i)] $d=3$ and $\Omega$ is equivalent to the $6$ equiangular lines described in section \ref{sec:examples}; or
\item[ii)] $|\Omega|>2d$ and $(\Gamma,\B_1)$ is a quasi-symmetric $2$-$(d,k,\lambda;s_1,s_2)$ design, where,
$k$ is a root of the quadratic
$$4x^2-4dx+(\rho-1)^2(d+\rho)$$
and 
$$\lambda=\frac{k(k-1)}{\rho^2-d},\,\,\,s_1=k-\frac{(\rho-1)^2}{4},\,\,\,s_2=k-\frac{(\rho^2-1)}{4}.$$
Moreover, $|\Omega|=d(d+1)/2$ if and only if $(\Gamma,\B_1)$ is a $3$-design.
\end{itemize}
\end{Thm}

\begin{proof}
By Proposition \ref{prop:tayeq}, $$k=|\Gamma_1(\gamma)|=\frac{d-\sqrt{d^2-(\rho-1)^2(d+\rho)}}{2}$$
for all $\gamma\in\Omega\backslash\Gamma$. Moreover,  as $|\Gamma_1(\gamma)|\neq |\Gamma_2(\gamma)|$, we can apply Theorem \ref{thm:linedes}. In particular,
$(\Gamma,\B_1)$ is a $2$-design, so by Fisher's inequality \cite[Theorem 1.14]{camvan}, $|\B_1|\geq d$. First suppose that $|\B_1|=d$, so $|\Omega|=2d$. 
Then as $\lambda=a(k-1)/2(d-k)$, by \eqref{eq:lam}, 
we deduce that 
$$a=\frac{2k(d-k)}{d-1}=\frac{(\rho-1)^2(d+\rho)}{2(d-1)}.$$ Further suppose that $\rho^2\leq d$. Then
$$a\leq\frac{(\sqrt{d}-1)^2(d+\sqrt{d})}{2(d-1)}=\frac{(d-\sqrt{d})}{2}.$$
However, this implies that
$$|\Omega|=2d\leq 3a\leq\frac{3(d-\sqrt{d})}{2}$$
which is a contradiction. Thus $\rho^2>d$ and, by Theorem \ref{thm:relbd}, 
$$|\Omega|=\frac{d(\rho^2-1)}{\rho^2-d}=2d$$
which gives $\rho^2=2d-1$. Therefore, in order for $|\Gamma_1(\gamma)|$ to be a real number,
$$d^2-(\rho-1)^2(d+\rho)=-d^2+4d-2\geq 0.$$
In particular, $2-\sqrt{2}\leq d\leq 2+\sqrt{2}$, so $d=1,2$ or $3$. Obviously $d\neq 1$ or $2$ as the maximum number of equiangular lines in these dimensions is $1$ and $3$ respectively. 
Thus $d=3$ and $|\Omega|=6$. The result i) now follows from the uniqueness of the regular two-graph on $6$ elements.

Now suppose that $|\B_1|>d$. By Proposition \ref{prop:tayint}, $(\Gamma,\B_1)$ has at most two block intersection numbers, and 
by \cite[Theorem 1.15]{camvan}, $(\Gamma,\B_1)$ has exactly one block intersection number if and only if $|\B_1|=d$. Thus
$\B_1$ has exactly two block intersection numbers, and hence is a $2$-$(d,k,\lambda;s_1,s_2)$ quasi-symmetric design, where
$$\lambda=\frac{a(k-1)}{2(d-k)},\,\,\,s_1=k-\frac{(\rho-1)^2}{4},\,\,\,s_2=k-\frac{(\rho^2-1)}{4}.$$
As $(\Gamma,\B_1)$ is quasi-symmetric design, we can apply Theorem \ref{thm:calder}. In particular $n=d$, 
$k=|\Gamma_1(\gamma)|$, $x=(\rho-1)^2/4$ and $y=(\rho^2-1)/4$ (Proposition \ref{prop:tayint}). Substituting these values into \eqref{eq:calder} gives
\begin{equation}\label{eq:fuse}f(d,k,(\rho-1)^2/4,(\rho^2-1)/4)=\frac{(\rho^2-1)^2(\rho^2-d-2)}{16}\geq 0,\end{equation}
so $\rho^2>d$. Thus by Theorem \ref{thm:relbd},
\begin{equation}\label{eq:sizeO}|\Omega|=\frac{d(\rho^2-1)}{\rho^2-d}=|\B_1|+d,\end{equation}
and hence
$$\lambda=\frac{k(k-1)}{\rho^2-d}.$$
Finally, the fact that $|\Omega|=d(d+1)/2$ if and only if $(\Gamma,\B_1)$ is a $3$-design follows from
Theorems \ref{thm:calder} and \ref{thm:absbd}, and \eqref{eq:fuse} and \eqref{eq:sizeO}.
\end{proof}

\section{Proofs of Theorem \ref{thm:main} and Theorem \ref{thm:main2}}\label{sec:proofmain}

We can now use Theorem \ref{thm:eqquasi} to prove Theorems \ref{thm:main} and \ref{thm:main2}.

\begin{proof}[Proof of Theorem \ref{thm:main}]
Let $\Omega$ be a set of $d(d+1)/2$ equiangular lines in $\R^d$ and $\Gamma\subseteq\Omega$ an incoherent subset of $d$ lines. 
The examples of sets of $d(d+1)/2$ equiangular lines in dimensions $d\leq 3$ are unique up to isometry, and they each have $\Inc(\Omega)=d$, so we 
can assume that $d>3$. 

By Proposition \ref{prop:tayeq}, for any $\gamma\in\Omega\backslash\Gamma$ we have that
$$|\Gamma_j(\gamma)|=\frac{d+(-1)^j\sqrt{d^2-(\rho-1)^2(d+\rho)}}{2}.$$
for $j=1,2$. As $|\Gamma_1(\gamma)|$ and $|\Gamma_2(\gamma)|$ are both integers we deduce that
$$\sqrt{d^2-(\rho-1)^2(d+\rho)}\in \Z$$
In particular there exists an integer $z$ such that
\begin{equation}\label{eq:zeq}z^2=d^2-(\rho-1)^2(d+\rho)\end{equation}
As $|\Omega|=d(d+1)/2$ we have, by Theorem \ref{thm:absbd}, that $d=\rho^2-2$. 
Substituting this into \eqref{eq:zeq} gives $$z^2=\rho^3-\rho^2-5\rho+6=(\rho-2)(\rho^2+\rho-3).$$
Therefore $(\rho,z)$ lies on the following elliptic curve:
\begin{equation}y^2=x^3-x^2-5x+6\end{equation}
Moreover as $d>3$, $\rho$ is an odd integer (Theorem \ref{thm:absbd}), and so $(\rho,z)$ is an integer point that lies on this elliptic curve. There are exactly
$13$ such points \cite[Elliptic Curve 156.a2]{lmfdb}, namely:
$$(-2,\pm 2), (-1,\pm 3), (1,\pm 1), (2,0), (3,\pm 3), (5,\pm 9),  (29, \pm 153)$$
Thus $\rho=1,3,5$ or $29$.
Clearly $\rho\neq 1$ as $d>0$. If $\rho=3$ or $5$ then $d=7$ or $23$ respectively. As the sets of $d(d+1)/2$ 
equiangular lines in dimensions $7$ and $23$ are unique up to isometry, in these cases $\Omega$ is 
isometric to corresponding examples given in Section \ref{sec:examples}.

Finally suppose that $\rho=29$, which implies that $d=839$, $|\Gamma_1(\gamma)|=343$ and $|\Gamma_2(\gamma)|=496$ for all $\gamma\in\Omega\backslash\Gamma$.
By Theorem \ref{thm:eqquasi}, $(\Gamma,\B_1)$ is a $2$-$(839,343,58653;147,133)$ quasi-symmetric design.
However, Theorem \ref{thm:eqquasi} also implies that $(\Gamma,\B_1)$ is a $3$-design, and hence  
$$\frac{343\times 342\times 341}{2\times 837}=\frac{71687}{3}$$
must be an integer, which is a contradiction.
\end{proof}

\begin{proof}[Proof of Theorem \ref{thm:main2}]
Let $X$ be a tight spherical $5$-design in $\S^{d-1}$. Then by \cite[Theorem 5.12]{delgoeseid}, $X$ is antipodal, that is $y\in X$ if and only if $-y\in X$. Hence
$X$ defines a set $\Omega$ of $d(d+1)/2$ lines through the origin in $\R^d$. It also follows from \cite[Theorem 5.12]{delgoeseid} that $\Omega$ is a set of equiangular lines.
If $X$ contains a subset $\Gamma$ of $d$ points whose pairwise inner products are positive, then the lines through the points in $\Gamma$ form
a maximal incoherent subset of $\Omega$ of size $d$. The result now follows from Theorem \ref{thm:main}.
\end{proof}

\section{Block sets and Equiangular lines}\label{sec:blocksets}

In this section we show that given a quasi-symmetric design with certain constraints on the parameters, 
one can construct a set of equiangular lines. This will help us prove Theorem \ref{thm:maineqquasi} in Section \ref{sec:thm3}.
However, we first describe a more general construction.

Let $\P$ be a point set of size $d$ and $\B$ be a set of $k$-sets (blocks) of elements from $\P$. If $(\P,\B)$ is such that
any two blocks intersect in either $s_1$ or $s_2$ elements, we show below that, given certain constraints, one can construct a set of equiangular lines from $(\P,\B)$. 
We call $(\P,\B$) a $(d,k;s_1,s_2)$ block set. We assume throughout that $s_1>s_2$, $\B$ is simple (i.e. contains no repeated blocks), 
and that there exists a pair of blocks whose intersection has exactly $s_j$ elements, for $j=1$ and $2$.

\begin{Const}\label{const:1} Let $(\P,\B)$ be a $(d,k;s_1,s_2)$ block set such that
\begin{equation}\label{eq:delta1}\Delta_1=k^2-\frac{d(s_1+s_2)}{2}\geq 0.\end{equation}
For each block $B\in\B$
let $v(B)$ be the vector in $\R^d$ given by
\begin{equation}\label{eq:vb}v(B)_j=\begin{cases}{d-k}+(-1)^{\epsilon_\B}\sqrt{\Delta_1}&\textnormal{ if $j\in B$ }\\
{-k}+(-1)^{\epsilon_\B}\sqrt{\Delta_1}&\textnormal{ if $j\in\P\backslash B$, }
\end{cases}\end{equation}
where $\epsilon_\B=0$ or $1$. We let $$\Omega(\B)=\{v(B)\,|\,B\in\B\}.$$
\end{Const}
For blocks $B_1,B_2\in\B$ with $|B_1\cap B_2|=s$, one calculates that the inner product
$$(v(B_1),v(B_2))=d(sd-k^2+\Delta_1),$$
and so
$$(v(B_1),v(B_2))=
\begin{cases}
d^2(2k-(s_1+s_2))/2&\textnormal{ if $B_1=B_2$}\\
d^2(s_1-s_2)/2&\textnormal{ if $|B_1\cap B_2|=s_1$}\\
-d^2(s_1-s_2)/2&\textnormal{ if $|B_1\cap B_2|=s_2.$}
\end{cases}$$
Thus we have the following.

\begin{Thm}\label{thm:blocksetconst}
Let $(\P,\B)$ be a $(d,k,s_1,s_2)$ block set such that $\Delta_1$ (given in \eqref{eq:delta1}) is non-negative. Then $\Omega(\B)$, given in Construction \ref{const:1}, is a set of vectors in $\R^d$
that span equiangular lines with common angle
$$\kappa=\frac{(s_1-s_2)}{2k-(s_1+s_2)}.$$
\end{Thm}

First note that the above theorem is independent of the parameter $\epsilon_\B$. This parameter is relevant in Theorem \ref{thm:const1} below. 
Also note that the vectors given in \eqref{eq:vb} can be defined for $\Delta_1<0$, in which case $v(B)\in\mathbf{C}^d$. However, in this 
case the corresponding lines spanned by these vectors are not necessarily equiangular. 

\begin{Ex}\label{ex:pairs}
Let $(\P,\B)$ be the set of all $2$-sets of $\P$. Then $k=2$, $s_1=1$ and $s_2=0$, so $\Delta_1\geq 0$ if and only if $d\leq 8$, and for these dimensions $\kappa=1/3$.
For $d=4,5$ we find maximal sets of $6$ and $10$ equiangular lines in the respective dimensions. For $d=6$, if we include the vector
$$v=(3,3,3,3,3,3)$$
with the $15$ vectors of $\Omega(\B)$, we find a maximal set of $16$ equiangular lines in $\R^6$. For $d=7$, we saw in Section \ref{sec:examples} 
that we can include $7$ extra vectors whose union with the $21$ vectors of $\Omega(\B)$ give $28$ equiangular lines in $\R^7$. 
Finally, for $d=8$ we find that $\Delta_1=0$. We thus deduce that each $v(B)$ is orthogonal to the all ones vector $\bf{1}\in\R^8$,
and so the set $\Omega(\B)$ of $28$ vectors spans a $7$-dimensional subspace. This representation of the $28$ equiangular lines in $\R^7$ is the one most often given in the literature.
\end{Ex}

\begin{Ex}\label{ex:276}Let $(\P,\B)$ be the $4$-$(23,7,1;3,1)$ quasi-symmetric design. We saw in Section \ref{sec:examples} that we can include $23$ vectors 
to the $253$ vectors of $\Omega(\B)$ to give a maximal set of $276$ equiangular line in $\R^{23}$. For any point $p\in\P$, the residual design $(\P\backslash\{p\},\B^p)$ is 
a $3$-$(22,7,4;3,1)$ quasi-symmetric design, and the corresponding set $\Omega(\B^p)$ gives a maximal set of $176$ equiangular lines in $\R^{22}$. 
\end{Ex}

\begin{Ex}\label{ex:d15} Let $(\P,\B)$ be a Steiner triple system of order $15$, which is a $2$-$(15,3,1;1,0)$ quasi-symmetric design that contains $35$ blocks. 
It is known that there are exactly $80$ non-isomorphic Steiner triple systems of order $15$ \cite{hallswift}. By including $$v=5\sqrt{\frac{3}{2}}{\bf{1}}\in\R^{15}$$
with the $35$ vectors of $\Omega(\B)$, we find a maximal set of $36$ equiangular lines in $\R^{15}$. 
\end{Ex}

\begin{Ex}\label{ex:janko} Let $J_2$ be the Hall-Janko-Wales group \cite{hall}, one of the sporadic simple groups.
This can be described as the automorphism group of the Hall-Janko graph $\G$ (see for example \cite{suzuki}). 
This group has two conjugacy classes of elements of order $3$ \cite{atlas}. Let $3A$ denote the conjugacy class of $J_2$ whose centraliser has order $1080$. 
It is known that in its action on $V\G$, the vertices of $\G$, for any $x\in 3A$ the number
of fixed points of $x$ is $10$ (see for example \cite[Section 3]{buek} or \cite[Remark 2.7]{selfdualcodes}). 
These vertices correspond to a coclique in $\G$, and by \cite[Lemma 2.8]{selfdualcodes} there are 280 cocliques in $\G$. Let $\B$ be the set of cocliques, that is
$$\B=\{\fix(x)\,|\,x\in 3A,\textnormal{ acting on $V\G$}\}$$
By \cite[Lemma 2.9]{selfdualcodes}, $(V\G,\B)$ is a $(100,10; 2,0)$ block set. In this case $\Delta_1=0$, so we know that the lines are in $\R^{99}$. However, using GAP \cite{gap} we find that
$\Omega(\B)$ spans a $63$-dimensional subspace. This is an example of a set lines that saturates the relative bound, 
that is $\Omega(\B)$ gives a maximal set of equiangular lines in $\R^{63}$ with common angle $\kappa=1/9$. We also note that $(V\G,\B)$ is not a $2$-design.
\end{Ex}

In Example \ref{ex:pairs} with $d=7$ and Example \ref{ex:276} with $d=23$, we were able to include an additional $d$ lines along with the lines constructed from the $(d,k;s_1,s_2)$ block set 
and still maintain the equiangular property. Indeed, the extra $d$ lines enabled us to saturate the incoherence bound. 
The next result shows that we can repeat this trick for certain parameter choices. First, however, we rewrite the parameters of
the designs in Theorem \ref{thm:maineqquasi}-iv) in terms of the intersection numbers.
\begin{Rem}\label{rem:intermsints}
Let $d,k,\rho, s_1,s_2$ and $\lambda$ be as in Theorem \ref{thm:maineqquasi}-iv).
By subtracting $s_2$ from $s_1$, we solve to give \begin{equation}\label{eq:rho}\rho=\rho(s_1,s_2)=2(s_1-s_2)+1\end{equation}
and by adding $s_1$ and $s_2$, we find
\begin{equation}\label{eq:k}k=k(s_1,s_2)=(s_1-s_2)^2+s_1.\end{equation}
Putting these values for $k$ and $\rho$ into \eqref{eq:taypoly} gives
\begin{equation}\label{eq:d}d=d(s_1,s_2)=\frac{(m^2+m+s_1)^2}{s_1}-2m\end{equation}
where $m=(s_1-s_2)$. Thus
\begin{equation}\label{eq:L}\lambda(s_1,s_2)=\frac{s_1(m^2+s_1)(m^2+s_1-1)}{s_1(2m(m+1)+1)-m^4-2m^3-s^2_2}.\end{equation}
\end{Rem}

\begin{Thm}\label{thm:const1} Let $(\P,\B)$ be a $(d,k;s_1,s_2)$ block set, where $d=d(s_1,s_2)$, $k=k(s_1,s_2)$ as in \eqref{eq:d}, \eqref{eq:k} respectively, and 
let $\Omega(\B)$ be the set of vectors constructed from $(\P,\B)$ as in Construction \ref{const:1}.
Then $\Omega(\B)$ spans a set of equiangular lines in $\R^d$. 
Moreover, for a specific choice of $\epsilon_\B$, there exist $d$ vectors in $\R^d$ whose union with $\Omega(\B)$ gives a set $(|\B|+d)$ vectors that span a set $\Omega$ of equiangular lines in $\R^d$
with common angle $\kappa=1/(2(s_1-s_2)+1)$ such that $\Inc(\Omega)=d$.
\end{Thm}

\begin{proof}
Let $d=d(s_1,s_2)$ and $k=k(s_1,s_2)$. Setting $m=s_1-s_2$ and substituting in the values for $d$ and $k$, one finds that
$$\Delta_1=\frac{m(m^2-s_2)^2}{2s_1}\geq 0,$$
so by Theorem \ref{thm:blocksetconst}, $\Omega(\B)$ spans a set of equiangular lines in $\R^d$ with common angle $\kappa=1/(2m+1)$.
However, whereas in Construction \ref{const:1} it did not matter if we chose $\epsilon_\B=0$ or $1$, for us to include an extra $d$ lines we have to be more precise. 
In particular, we let $v(B)$ be the vector in $\R^d$ such that
$$v(B)_j=\begin{cases}{d-k}+\Delta&\textnormal{ if $j\in B$ }\\
{-k}+\Delta&\textnormal{ if $j\in\P\backslash B$, }
\end{cases}$$
where $$\Delta=\sqrt{m}(m^2-s_2)/\sqrt{2s_1}.$$ We observe that $\Delta<0$ if and only if $m^2<s_2$, which in turn holds if and only if $k>d/2$.
For $B_1,B_2\in\B$ we calculate that 
$$(v(B_1),v(B_2))=
\begin{cases}
d^2m(2m+1)/2&\textnormal{ if $B_1=B_2$}\\
d^2m/2&\textnormal{ if $|B_1\cap B_2|=s_1$}\\
-d^2m/2&\textnormal{ if $|B_1\cap B_2|=s_2.$}
\end{cases}$$

Now for $i\in\P$ let $v(i)$ denote the vector in $\R^d$ given by
$$v(i)_j=\begin{cases}(s_1-s_2)(d-1)-\sqrt{\Delta_2}&\textnormal{ if $j=i$ }\\
-(s_1-s_2)-\sqrt{\Delta_2}&\textnormal{ if $j\neq i$ }
\end{cases}$$
where $$\Delta_2=\frac{m(2m+d)}{2}.$$
One calculates that $$(v(i),v(i))=d(m^2(d-1)+\Delta_2)=d^2m(2m+1)/2,$$
and for distinct $i,j\in\P$,
$$(v(i),v(j))=d(\Delta_2-m^2)=d^2m/2.$$
Finally, for $i\in\P$ and $B\in\B$, with $s=|B\cap\{i\}|$, 
$$(v(B),v(i))=m(sd-k)d-d\Delta\sqrt{\Delta_2}=
\begin{cases}
d^2m/2&\textnormal{ if $s=1$}\\
-d^2m/2&\textnormal{ if $s=0$.}
\end{cases}$$
It follows that $\Omega=\{v(B)\,|\,B\in\B\}\cup\{v(i)\,|\,i\in\P\}$ is a set of $|\B|+d$ vectors that span equiangular lines in $\R^d$ with common angle $\kappa=1/(2m+1)$
and $\Gamma=\{v(i)\,|\,i\in\P\}$ is an incoherent subset of size $d$.
\end{proof}

\begin{Cor}\label{cor:thmcon} If in addition $(\P,\B)$ is a quasi-symmetric design where every pair is in 
$\lambda(s_1,s_2)$ blocks (given in \eqref{eq:L}) then $|\Omega|$ saturates
the relative bound. In particular, $(\Omega,\C)$ is a regular two-graph. 
\end{Cor}
\begin{proof}
Let $d=d(s_1,s_2)$, $k=k(s_1,s_2)$ and $\lambda=\lambda(s_1,s_2)$. As $(\P,\B)$ is a quasi-symmetric design, Theorem \ref{thm:calder} states that
$$f(d,k,x,y)\geq 0$$
where $f$ is given in \eqref{eq:calder} and $x=k-s_1$, $y=k-s_2$. Letting $\rho=\kappa^{-1}=2m+1$ and substituting in the values for $d,k,x,y$,
we calculate that $f(d,k,x,y)$ is equal to the expression given in \eqref{eq:fuse}. Thus $\rho^2>d$, and one further calculates that
$$|\Omega|=\frac{d(d-1)}{k(k-1)}\lambda+d=\frac{d(\rho^2-1)}{\rho^2-d}.$$
Theorem \ref{thm:relbd} now implies that $(\Omega,\C)$ is a regular two-graph.
\end{proof}

\section{Proof of Theorem \ref{thm:maineqquasi}}\label{sec:thm3}

In Sections \ref{sec:maxinc} and \ref{sec:blocksets} we proved most of the results required to prove Theorem \ref{thm:maineqquasi}. 
However, before we can complete the proof in this section, we must deal with
the case where $|\Gamma_1(\gamma)|=d/2$ for all $\gamma\in\Omega\backslash\Gamma$. 
In Theorem \ref{thm:eqquasi} we dealt with the case $|\Gamma_1(\gamma)|<d/2$ for all $\gamma\in\Omega\backslash\Gamma$. This allowed us to apply Theorem \ref{thm:linedes} 
and construct a quasi-symmetric design on $\Gamma$. 
When $|\Gamma_1(\gamma)|=d/2$, we are unable to do this. This is in part because we cannot distinguish between $\Gamma_1(\gamma)$ and $\Gamma_2(\gamma)$ 
by considering their size alone. However, we are still able to find a quasi-symmetric design in this case.

\begin{Thm}\label{thm:deq2k}Let $\Omega$ be a set of equiangular lines in $\R^d$ such that $(\Omega,\C)$ is a regular two-graph with parameters $(n,a,b)$
and $\Inc(\Omega)=d$, and let $\Gamma\subseteq\Omega$ be an incoherent subset of $d$ lines and $\rho=\kappa^{-1}$. If
$|\Gamma_1(\gamma)|=d/2$ for some $\gamma\in\Omega\backslash\Gamma$, then $d=\rho(\rho-1)$, $|\Omega|=(\rho^2-1)(\rho-1)$, and either
\begin{itemize}
\item[i)] $\rho=2$ and $\Omega$ is isometric to the set of $3$ equiangular lines in $\R^2$; or
\item[ii)] $(\Gamma,\B)$ is a $3$-$(d,d/2,n-d-3a/2)$ design, where
$$\B=\{\Gamma_i(\gamma)\,|\,\gamma\in\Omega\backslash\Gamma, i=1,2\};$$ for $\alpha\in\Gamma$, the derived design $(\Gamma\backslash\{\alpha\},\B_\alpha)$ 
is a quasi-symmetric $$2-(d-1,d/2-1,n-d-3a/2;s_1-1,s_2-1)$$ design; and the residual design $(\Gamma\backslash\{\alpha\},\B^\alpha)$ is a quasi-symmetric 
$$2-(d-1,d/2,a/2;s_1,s_2)$$ design,
where \begin{equation}\label{eq:s1s2}s_1=\frac{(\rho^2-1)}{4},\,\,\,s_2=\frac{(\rho-1)^2}{4}.\end{equation}
\end{itemize}
\end{Thm}

\begin{proof}
Since $|\Gamma_1(\gamma)|=d/2$, it follows from Proposition \ref{prop:tayeq} that 
$$d^2-(\rho-1)^2(d+\rho)=0.$$
In particular, $d$ is one of the solutions to the equation
$$x^2-(\rho-1)^2x-\rho(\rho-1)^2=0,$$
which are $\rho(\rho-1)$ or $(1-\rho)$. As $d> 1$, we conclude that $d=\rho(\rho-1)$ and hence $\rho^2>d$. 
It now follows from Theorem \ref{thm:relbd} that $|\Omega|=(\rho-1)(\rho^2-1)$. Clearly if $\rho=2$ then i) holds.

We now suppose that $\rho>2$, so $d>2$. Let $\{\alpha,\beta,\gamma\}$ be a $3$-set of elements in $\Gamma$. We deduce from Theorem \ref{thm:desint1} that
$$S_{\alpha\beta}\cup S_{\alpha\gamma}=S_{\alpha\beta}\cup S_{\beta\gamma}=S_{\alpha\gamma}\cup S_{\beta\gamma}$$
and
$$|S_{\alpha\beta}\cup S_{\alpha\gamma}|=\frac{3a}{2}.$$
Furthermore, by Lemma \ref{lem:simple}, for each $\delta\in S_{\alpha\beta}\cup S_{\alpha\gamma}$, the set $\{\alpha,\beta,\gamma\}$ has a non-trivial
intersection with both $\Gamma_1(\delta)$ and $\Gamma_2(\delta)$. Additionally, for each $\delta\in\Omega\backslash(\Gamma\cup S_{\alpha\beta}\cup S_{\alpha\gamma})$,
$\{\alpha,\beta,\gamma\}$ is entirely contained in either $\Gamma_1(\delta)$ or $\Gamma_2(\delta)$. We conclude that $(\Gamma,\B)$ is a $3$-$(d,d/2,n-d-3a/2)$ block design.

Let $\alpha\in\Gamma$ and consider the derived design $(\Gamma\backslash\{\alpha\},\B_\alpha)$,
where $$\B_\alpha=\{B\backslash\{\alpha\}\,|\,B\in\B\textnormal{ and }\alpha\in B\}.$$
This is a $2$-$(d-1,d/2-1,n-d-3a/2)$ design \cite[Definition 1.32]{camvan}.
Now consider distinct elements $B_1,B_2\in\B$. As $|B_1|=|B_2|=d/2$, it follows that $|B_1\cap B_2|=0$ if and only if $B_2=\Gamma\backslash B_1$.
As $s_1+s_2=d/2$, we deduce from Proposition \ref{prop:tayint} that $|B_1\cap B_2|=0,s_1$ or $s_2$, where $s_1,s_2$ are as in \eqref{eq:s1s2}.
We therefore conclude that two blocks of $\B_\alpha$ have either $s_1-1$ or $s_2-1$ elements in common, proving the result for the derived design. 
Similarly, for $\alpha\in\Gamma$, one deduces that the residual design $(\Gamma\backslash\{\alpha\},\B^\alpha)$ with respect to $\alpha$, where
$$\B^\alpha=\{B\,|\,B\in\B\textnormal{ and }\alpha\notin B\},$$
is a quasi-symmetric $2$-$(d-1,d/2,a/2;s_1,s_2)$ design.
\end{proof}

\begin{Rem}\label{rem:paramsdes}
We can describe the designs that appear in Theorem \ref{thm:deq2k}-ii) in terms of one parameter. 
In this case $d=\rho(\rho-1)\geq 4$, where $\rho=\kappa^{-1}$, and so $|\Omega|>2d$. Thus, by Theorem \ref{thm:rhoodd}, $\rho$ is an odd integer. Let $$\rho=2i+1$$
for some integer $i$. Then $$d=2i(2i+1),\,\, n=8i^2(i+1),\,\,\,s_1=i^2+i,\,\,\,s_2=i^2.$$ 
Moreover, it follows that \eqref{eq:setsum2} is equal to $(n-d)d/2$, which we solve to give
$$a=\frac{d(n-d)}{2(d-1)}=2i^2(2i+1).$$
Therefore the design $(\Gamma,\B)$ is a 
\begin{equation}\label{eq:conv1}3-(2i(2i+1),i(2i+1),i(2i^2+i-2))\end{equation} 
design; the derived design is a 
\begin{equation}\label{eq:conv2}2-(2i(2i+1)-1,(2i-1)(i+1),i(2i^2+i-2);i^2+i-1,i^2-1)\end{equation} 
quasi-symmetric design; and the residual design is a
\begin{equation}\label{eq:conv3}2-(2i(2i+1)-1,i(2i+1),i^2(2i+1); i^2+i,i^2)\end{equation} 
quasi-symmetric design.
\end{Rem}

\begin{Thm}\label{thm:desdeq2k} Suppose a quasi-symmetric with the parameters in \eqref{eq:conv2} exists for some integer $i\geq 1$. 
Then there exists a set $\Omega$ of $(\rho^2-1)(\rho-1)$ equiangular lines in $\R^d$ with common angle $\kappa=\rho^{-1}$ and $\Inc(\Omega)=d$, 
where $d=\rho(\rho-1)$ and $\rho=2i+1$. In particular, $\Omega$ saturates the relative bound, the incoherence bound, and $(\Omega,\C)$ is a regular two-graph.
\end{Thm}

\begin{proof}
Let $(\P,\B)$ be a quasi-symmetric design with the parameters in \eqref{eq:conv2}. 
Let $\alpha$ be an extra point and consider the block set $(\P\cup\{\alpha\},\hat{\B})$ where
$$\hat{\B}=\{B\cup\{\alpha\}\,|\, B\in\B\}.$$
This is a $(2i(2i+1),i(2i+1);i^2+i,i^2)$ block set, and in particular, 
$$2i(2i+1)=d(s_1,s_2),\,\,\,i(2i+1)=k(s_1,s_2),$$ 
where $s_1=i^2+i$, $s_2=i^2$, and $d(s_1,s_2)$, $k(s_1,s_2)$ 
are as in \eqref{eq:d} and \eqref{eq:k}. Thus we can apply Theorem \ref{thm:const1} to $(\P\cup\{\alpha\},\hat{\B})$ giving a set $\Omega$ of $(|\hat{\B}|+d)=(\rho^2-1)(\rho-1)$ 
equiangular lines in $\R^d$ with common angle $\kappa$ such that $\Inc(\Omega)=d$, where $d=\rho(\rho-1)$ and $\kappa^{-1}=\rho=2i+1$. 
One calculates that $|\Omega|$ saturates the relative bound, and as $\rho^2>d$, $(\Omega,\C)$ is therefore a regular two-graph by Theorem \ref{thm:relbd}.
\end{proof}

\begin{proof}[Proof of Theorem \ref{thm:maineqquasi}] We deduce that Theorem \ref{thm:maineqquasi} is a direct consequence of Theorem \ref{thm:relbd},
Theorem \ref{thm:eqquasi}, Theorem \ref{thm:const1}, Corollary \ref{cor:thmcon}, 
Theorem \ref{thm:deq2k} and Theorem \ref{thm:desdeq2k}.
\end{proof}

\section{Proof of Theorem \ref{thm:main4}}\label{sec:main4}

Up to isometry, the only known sets of equiangular lines that saturate the relative and incoherence bounds are the sets of lines given in
Section \ref{sec:examples}. Let us look at the examples in dimensions $d\geq 4$ in more detail. 

\begin{Ex}\label{ex:16lines} Let $\Omega$ be the set of $16$ equiangular lines in $\R^{6}$. Then the regular two-graph $(\Omega,\C)$ can be identified with
the symplectic regular two-graph on $V$, the $4$-dimensional vector space over the binary field $\mF_2$ \cite{taylor}. 
Given a non-degenerate alternating bilinear form $\varphi$ on $V$, 
the symplectic regular two-graph on $V$ is $(V,\C_V)$, where $$\{v_1,v_2,v_3\}\in \C_V\iff\varphi(v_1,v_2)+\varphi(v_1,v_3)+\varphi(v_2,v_3)=0.$$
The automorphism group of this two-graph is isomorphic to $N:\Sp(4,2)$ acting $2$-transitively on $V$, 
where $N$ is the group of translations of $V$ and $\Sp(4,2)$ is the group generated by the transvections of $V$. Let $\Gamma$ be a maximal incoherent 
$6$-set of $V$. Then without loss of generality we can assume that $$\Gamma=\{{\bf{0}},v_1,v_2,v_3,v_4,v_5\}$$
where $\varphi(v_i,v_j)=1$ for all $i\neq j$. Let $H$ be the stabiliser in $N:\Sp(4,2)$ of $\Gamma$. One calculates that for each $v\in\Gamma$,
$n_vt_v\in H$, where for all $u\in V$
$$n_v(u)=u+v,\,\,\,\,t_v(u)=u+\varphi(v,u)v.$$
Furthermore, $n_vt_v$ interchanges $\bf{0}$ and $v$ and fixes all other elements in $\Gamma$. Thus the induced action of $H$ on $\Gamma$ is isomorphic to
the natural action of the symmetric group $\Sym(6)$ on six points. Thus it follows that
$$|S_{v_{i_1}v_{i_2}}\cap S_{v_{i_3}v_{i_4}}|$$ is constant
on all $4$-sets $\{v_{i_1},v_{i_2},v_{i_3},v_{i_4}\}$ of elements in $\Gamma$.
\end{Ex}

\begin{Ex}
Let $\Omega$ be the maximal set of equiangular lines in $\R^d$, for $d=7$ or $23$, and let $\Gamma$ be a maximal incoherent subset of $d$ lines.
In the proof of Theorem \ref{thm:eqquasi} we showed that $(\Gamma,\B_1)$ is a quasi-symmetric design; for $d=7$ it is the $2$-$(7,2,1;1,0)$ design and for $d=23$ it
is the $2$-$(23,7,21;3,1)$ design.
Let $G$ be the automorphism group of $(\Gamma,\B_1)$. Then it is well known that $G\cong M_{23}$, the Mathieu Group on $23$ points, if $d=23$ (see for example \cite[Theorem 6.7B]{dixmort}), and it is straightforward to see that $G\cong\Sym(7)$ for $d=7$. 
In both cases $G$ acts $4$-transitively on $\Gamma$.
We thus conclude that $$|S_{\alpha\beta}\cap S_{\eta\nu}|$$ is constant on all $4$-subsets $\{\alpha,\beta,\eta,\nu\}$ of $\Gamma$.
\end{Ex}

Thus for the sets of equiangular lines in the above examples,
$$|S_{\alpha\beta}\cap S_{\eta\nu}|$$ is constant for all $4$-sets of a maximal incoherent set $\Gamma$ of size $d$.
Interestingly, the following result holds in the general case.

\begin{Lem}\label{lem:4setsum} 
Let $\Omega$ be a set of equiangular lines in $\R^d$
such that $(\Omega,\C)$ is a regular two-graph with parameters $(n,a,b)$ and $\Inc(\Omega)=d$, and let $\Gamma\subseteq\Omega$ is a subset of $d$ incoherent lines.
Then for $\alpha,\beta\in\Gamma$, 
\begin{equation}\label{eq:4setsum}\sum_{\eta,\nu\in\Gamma\backslash\{\alpha,\beta\}}|S_{\alpha\beta}\cap S_{\eta\nu}|=a(d-k-1)(k-1).\end{equation}
where $k=|\Gamma_1(\gamma)|$ for any $\gamma\in\Omega\backslash\Gamma$. 
\end{Lem}

\begin{proof}
Let $M$ be the $(n-d)\times \binom{d}{2}$ matrix whose columns are the characteristic vectors of the sets $S_{\alpha\beta}$, where $\alpha,\beta\in\Gamma$. That is, the rows of $M$ are 
labelled by $\Omega\backslash\Gamma$, the columns labelled by pairs of elements of $\Gamma$, and
$$M_{\gamma,\{\alpha\beta\}}=\begin{cases}{1}&\textnormal{ if $\gamma\in S_{\alpha\beta}$ }\\
0&\textnormal{ if $\gamma\notin S_{\alpha\beta}$ }.
\end{cases}$$
As $(\Omega,\C)$ is a regular two-graph, the columns of $M$ have $a$ non-zero entries, and moreover, the rows of $M$ have $k(d-k)$ non-zero entries.
One deduces the $\bf{1}$ is an eigenvector of $M^TM$ with eigenvalue $k(d-k)a$.

Now observe that the ($\{\alpha\beta\},\{\eta\nu\}$) entry of $M^TM$ is equal to 
$$|S_{\alpha\beta}\cap S_{\eta\nu}|.$$ Since $\bf{1}$ is an eigenvector with eigenvalue $k(d-k)a$, $|S_{\alpha\beta}|=a$, and by Lemma \ref{thm:desint1}, 
$$|S_{\alpha\beta}\cap S_{\alpha\eta}|=|S_{\alpha\beta}\cap S_{\beta\eta}|=a/2$$ for all $\eta\in\Gamma\backslash\{\alpha,\beta\}$, we deduce that \eqref{eq:4setsum} holds.
\end{proof}

We can now prove Theorem \ref{thm:main4}.

\begin{proof}[Proof of Theorem \ref{thm:main4}] 
Suppose that $d\geq 4$ and $$|S_{\alpha\beta}\cap S_{\eta\nu}|=c$$ for all $4$-sets $\{\alpha,\beta,\eta,\nu\}$ of $\Gamma$, where $c$ is some constant. 
If $(\Omega,\C)$ has parameters $(n,a,b)$ and $k=|\Gamma_1(\gamma)|$ for some $\gamma\in\Omega\backslash\Gamma$, then by Lemma \ref{lem:4setsum}, 
$$c=\frac{2a(d-k-1)(k-1)}{(d-2)(d-3)}.$$
Hence we deduce from Theorem \ref{thm:desint1} that 
$$M^TM=aI+\frac{a}{2}A+c(J-I-A)$$
where $M$ is the matrix defined in the proof of Lemma \ref{lem:4setsum} and $A$ is the adjacency matrix of the Triangular graph $T(d)$ (see \cite[Chapter 9]{spectra} 
for the definition of the Triangular graph). It follows that the eigenvalues of $M^TM$ are
\begin{align*}
\theta_0&=a(d-1)+\frac{c(d-2)(d-3)}{2}=k(d-k)a\\
\theta_1&=\frac{a(d-2)}{2}-c(d-3)\\
\theta_2&=c
\end{align*}
with multiplicities $$m_{\theta_0}=1,\,\,m_{\theta_1}=d-1,\,\,m_{\theta_2}=\frac{d(d-3)}{2}.$$

If $n-d=d(d-1)/2$, then by Theorem \ref{thm:main}, $d=7$ or $23$ and $\Omega$ is isometric to the corresponding example given in Section \ref{sec:examples}. 

Suppose now that $n-d<d(d-1)/2$. 
As $MM^T$ has the same non-zero eigenvalue spectrum as $M^TM$, we deduce that one of the eigenvalues of $M^TM$ must be equal to zero. 
It is straightforward to show that $c\neq 0$, and so we must have that $\theta_1=0$. We thus deduce that
$$4(d-k-1)(k-1)=(d-2)^2$$
Suppose $2k<d$, and substitute the values for $d$ and $k$ given in Remark \ref{rem:intermsints} in the expression above. We conclude that $m^2=s_2$, from which we deduce
that $2k=d$, a contradiction.

Thus $2k=d$, and by Remark \ref{rem:paramsdes}, 
$$d=2i(2i+1),\,\, n=|\Omega|=8i^2(i+1)$$
for some positive integer $i$. Moreover, we have that $MM^T$ has eigenvalues $\theta_0$ and $\theta_2$ with respective multiplicities $1$ and $d(d-3)/2$. 
However, the multiplicity of $\theta_2$ must also equal $n-d-1$, so
$$d(d-3)/2=n-d-1.$$
Substituting in the values for $d$ and $n$, we deduce that this holds if and only if $i=1$. Therefore $d=6$ and $\Omega$ is isometric to the corresponding
set of lines given in Section \ref{sec:examples}
\end{proof}

\subsection{Two Conjectures}
By Theorem \ref{thm:absbd}, we know that if a set equiangular lines in $\R^d$ ($d>3$) saturates the absolute bound, $d+2$ is the square of an odd integer. 
Bannai et al.~\cite{bannaivenkov}, and subsequently Nebe and Venkov \cite{nebevenkov}, proved 
that no such set of lines exists for infinitely many feasible values of $d$. This, along with Theorem \ref{thm:main}, 
is suggestive that the only dimensions in which the absolute bound is saturated are the known ones.
Thus we make the following conjecture, which if correct will, along with Theorem \ref{thm:main}, classify the sets of equiangular lines that saturate the absolute upper bound.

\begin{Conj}\label{conj:abs} Let $\Omega$ be a set of $d(d+1)/2$ equiangular lines in $\R^d$. Then $\Inc(\Omega)=d$.
\end{Conj}

Recall from Theorem \ref{thm:desint1} that for any set of equiangular lines such that $(\Omega,\C)$ is a regular two-graph,
$|S_{\alpha\beta}\cap S_{\alpha\gamma}|$ is constant for all incoherent $3$-sets. Moreover,
Lemma \ref{lem:4setsum} shows that the sum of \begin{equation}\label{eq:constssets}|S_{\alpha\beta}\cap S_{\eta\nu}|\end{equation} over all pairs disjoint from $\alpha,\beta$ in a maximal
incoherent $d$-set is constant. Therefore it seems plausible that \eqref{eq:constssets} is constant for all $4$-sets of a maximal incoherent subset of $\Omega$.
Thus given Theorem \ref{thm:main4}, we conjecture the following.

\begin{Conj}\label{conj:relinc} Let $\Omega$ be a set of equiangular lines that saturates the relative bound and incoherence bound. Then $d=2,3,6,7,23$ and
$\Omega$ is isometric to the corresponding example found in Section \ref{sec:examples}.
\end{Conj}

\section{The $276$ Equiangular lines in $\R^{23}$ and the roots of $E_8$}\label{sec:rootsE8}

We now identify the roots of the $E_8$ lattice with a subset of $\Omega$, the $276$ equiangular lines in $\R^{23}$. The trivial observation that
$$276=240+36$$
is suggestive. In particular, in Proposition \ref{prop:36lines} we identify a subset of $36$ of the lines in $\Omega$ and show that they form a maximal set of equiangular lines in a $15$-dimensional subspace of $\R^{23}$. Then in Theorem \ref{thm:rootsE8}, we show that the projection the lines in $\Omega$ (or in particular, a set of unit vectors representing $\Omega$) 
onto the orthogonal complement of this $15$-dimensional subspace gives the roots of $E_8$. 

Before we turn our attention to proving these results, we must first understand in more detail the $S(4,7,23)$ Steiner system used in Section \ref{sec:examples} (also given 
in Theorem \ref{thm:eqquasi})
to describe the $276$ equiangular lines in $\R^{23}$. The blocks of the $S(4,7,23)$ Steiner system are often referred to as a \emph{heptads}. 
The automorphism group $G$ of this Steiner system is isomorphic to the sporadic simple group $M_{23}$ \cite[Theorem 6.7B]{dixmort}, and in the following lemma
we collect certain facts about the stabiliser of a heptad in $G$. These facts can be determined from the cited results and the observation that the $S(4,7,23)$ Steiner system
can be seen as the derived design of the $S(5,8,24)$ Steiner system.

\begin{Lem}\label{lem:facts}(see \cite[Theorems 10 and 14, Chapter 10]{conwaysloane} and \cite[Theorem 2.10.1]{ivanov}) Let $H$ be the stabiliser in $G$ of a heptad $B$. 
Then $H\cong 2^4:A_7$ and the kernel of the action of $H$ on $B$ is isomorphic to the elementary abelian group $2^4$, 
which acts regularly on the complement of the heptad. Moreover, $H$ has a faithful $3$-transitive action on the complement of $B$.
\end{Lem}

Given an involution $x$ in $G\cong M_{23}$ we say, with respect to $x$, a heptad $B$ is of
\begin{itemize}
\item[i)] Type 0 if $\fix(x)=B$;
\item[ii)] Type $1$ if $|B\cap\supp(x)|=4$ and $B^x=B$;
\item[iii)] Type $2$ if $|B\cap\supp(x)|=4$ and $B^x\neq B$; or
\item[iv)] Type $3$ if $|B\cap\supp(x)|=6$.
\end{itemize}
Here $\fix(x)$ denotes the fixed points of $x$ and $\supp(x)$ the support, or moved points, of $x$. 
 
\begin{Prop}\label{prop:invtypes} Let $x$ be an involution in $G$. Then, with respect to $x$, there is exactly $1$ heptad of Type $0$; $28$ heptads of Type $1$;
$112$ heptads of Type $2$; and $112$ heptads of Type $3$. Moreover, if $B$ is a heptad of Type $2$ or $3$ with respect to $x$, then $|B\cap B^x|=3$. 
\end{Prop}

\begin{proof}
Let $H$ be the set-wise stabiliser in $G$ of a heptad $B$. 
By Lemma \ref{lem:facts}, $H\cong 2^4:A_7$, and the kernel $K$ of the action of $H$ on $B$ is isomorphic to
the elementary abelian group $2^4$.  Thus, for each non-trivial $x\in K$, $\fix(x)=B$ is a heptad. As there is only one conjugacy class of involutions in $M_{23}$ \cite{atlas}, 
it follows that for any involution $x\in G$, $\fix(x)$ is a heptad of Type 0 with respect to $x$.

Now let $x\in G$ be an involution and $B=\fix(x)$, the heptad of Type 0 with respect to $x$. Using simple counting arguments, one can deduce that any other heptad intersects $B$ in $1$ or $3$ points. 
In particular, there are exactly $112$ heptads that intersect $B$ in $1$ point, and exactly $140$ heptads that intersect $B$ in $3$ points. 

Let $\hat{B}$ be one of the heptads that intersect $B$ in $3$ points, so $x$ fixes point-wise $3$ of the elements of
$\hat{B}$ and moves the remaining four. As $x$ is has no fixed points in the complement of $B$, it is the product of $8$ transpositions. Let $S$ be the set of these $8$ transpositions. 
Then the $4$ points of $\hat{B}$ moved by $x$ are contained the support of either $2,3$ or $4$ transpositions in $S$. Suppose the $4$ points are contained in the support
of $3$ transpositions. Then this implies that exactly $2$ of the four points are interchanged, whereas the remaining two points are moved to points not contained in $\hat{B}$. However,
this would imply that $|\hat{B}\cap\hat{B}^x|=5$, contradicting the fact that $S(4,7,23)$ is a Steiner system. Therefore the $4$ points moved by $x$ are contained in the support
of either $2$ or $4$ transpositions in $S$. 

As $H$ acts $3$-transitively on the support of $x$, for each pair of transpositions $t_i, t_k\in S$ there exists a heptad $\hat{B}$ such
that $\hat{B}\cap\supp(x)=\supp(t_it_k)$. For such a heptad it is clear that $\hat{B}^x=\hat{B}$, and so it is of Type $1$.
For each other heptad $\hat{B}$ that intersect $B$ in $3$ points, the $4$-points moved by $x$ are contained in the support of $4$ transpositions in $S$, and so $\hat{B}^x\neq\hat{B}$. 
Thus there are exactly $\binom{8}{2}=28$ heptads of Type $1$ and $140-28=112$ of Type $2$. It is straightforward to deduce that $|\hat{B}\cap\hat{B}^x|=3$ for a heptad of Type $2$.

The remaining $112$ heptads intersect $B$ in exactly $1$ point, and so are of Type $3$. Let $\hat{B}$ be one of these heptads and
$T$ be the number of transpositions in $S$ whose support contains at least one of the six points in $\hat{B}\cap\supp(x)$. Then $T=3,4,5$ or $6$.
For each pair of transpositions $t_i,t_k\in S$, we saw that there exists a unique heptad $\tilde{B}$ such that $\tilde{B}\cap\supp(x)=\supp(t_it_k)$, 
which implies that $T\neq 3$ or $4$. Now suppose that $T=6$, and let $t_1,\ldots,t_6$ be the six transpositions in $S$
whose support contains an element of $\hat{B}$. As before, for $j=2,\ldots 6$ there exists a unique heptad $B(j)$ such that $B(j)\cap\supp(x)=\supp(t_1t_j)$. Furthermore,
one deduces that $|B(j)\cap\hat{B}|=3$, and that each $B(j)$ intersects $B$ in a common point, namely $\{p\}=B\cap\hat{B}$. In particular, $B(j)\cap B(k)=\{p\}\cup\supp(t_1)$ for all $j\neq k$, and moreover,
$|B(j)\cap B|=3$ for all $j$. However, this implies that $5$ sets of size $2$ must fit into a set of size $6$ without overlapping, which is not possible. Hence $T=5$
and one deduces that $|\hat{B}\cap \hat{B}^x|=3$.
\end{proof}

Let $\Omega$ be the set of $276$ equiangular lines in $\R^{23}$ and let $\Gamma\subseteq\Omega$ be a maximal incoherent subset of $23$ lines.
We now choose unit vectors to represent the lines in $\Omega$. Let $\mathfrak{B}=\{\alpha_i\}_{i=1}^{23}$ be a set of unit vectors that represent
the lines in $\Gamma$ such that $(\alpha_i,\alpha_j)=1/5$ for all $i\neq j$. For each $\gamma\in\Omega\backslash\Gamma$ we now 
identify $\Gamma_1(\gamma)$ with the set of unit vectors in $\mathfrak{B}$ corresponding to the lines in $\Gamma_1(\gamma)$, and similarly for $\Gamma_2(\gamma)$. 
By Remark \ref{rem:coeff}, for each $\gamma\in\Omega\backslash\Gamma$, the unit vector
\begin{equation}\label{eq:vecsdef}\frac{1}{3}\sum_{\alpha_i\in\Gamma_1(\gamma)}\alpha_i-\frac{1}{6}\sum_{\alpha_i\in\Gamma_2(\gamma)}\alpha_i\end{equation}
represents $\gamma$. In fact, we shall use $\gamma$ to denote the unit vector above.
Let $\mathcal{U}$ be the union of these vectors with the basis $\mathfrak{B}$, so $\U$ is a set of unit vectors that represents the lines in $\Omega$.

Clearly any permutation of $\Gamma$ induces a permutation of $\mathfrak{B}$, and so the automorphism group $G\cong M_{23}$ of $(\Gamma,\B_1)$ 
as given in Theorem \ref{thm:eqquasi} induces an action on $\U$. 

\begin{Prop}\label{prop:36lines} Let $G$ be as above and let $x\in G$ be an involution. Then $x$ fixes point-wise $36$ elements of $\U$. Let $V$ be the span
of these $36$ vectors. Then $\dim(V)=15$, and the two-graph of the corresponding $36$ equiangular lines is regular.
\end{Prop}

\begin{proof}
With respect to the basis $\mathfrak{B}$, $x$ is a permutation matrix, and its characteristic polynomial is equal to $(\lambda-1)^{15}(\lambda+1)^8$ as it the product of $8$ transpositions. 
Therefore $V$, the $1$-eigenspace of $x$, is a $15$-dimensional subspace and $W$, the $(-1)$-eigenspace of $x$, is an $8$-dimensional subspace.
We claim that exactly $36$ lines of $\U$ are contained in $V$. 

By Proposition \ref{prop:invtypes}, the fixed points of $x$ are a heptad of the $S(4,7,23)$ Steiner system. Let $\gamma\in\U\backslash\mathfrak{B}$ be the vector
such that $\Gamma_1(\gamma)$ are the fixed points of $x$. Clearly $\alpha_i\in V$ for $\alpha_i\in\Gamma_1(\gamma)$.
Furthermore, $\Gamma_i(\gamma)^x=\Gamma_i(\gamma)$ for $i=1,2$, and so
$$\gamma^x=1/3\sum_{\alpha_i\in\Gamma_1(\gamma)}\alpha^x_i-1/6\sum_{\alpha_i\in\Gamma_2(\gamma)}\alpha^x_i=1/3\sum_{\alpha_i\in\Gamma_1(\gamma)}\alpha_i-1/6\sum_{\alpha_i\in\Gamma_2(\gamma)}\alpha_i=\gamma.$$
Finally, in Proposition \ref{prop:invtypes} we saw that there exist exactly 
$28$ heptads of Type $1$ with respect to $x$, that is, $28$ heptads $\Gamma_1(\delta)$ such that $|\Gamma_1(\delta)\cap \Gamma_1(\gamma)|=3$ and 
$\Gamma_1(\delta)^x=\Gamma_1(\delta)$ (and so $\Gamma_2(\delta)^x=\Gamma_2(\delta)$). 
In particular, a similar calculation to above shows that $\delta^x=\delta$ for each of these heptads of Type $1$.
Thus we have
$$7+1+28=36$$
vectors in $V$. These $36$ vectors represent $36$ equiangular lines in $\Omega$ contained in $V$ with common angle $\kappa=1/5$, and it follows from
Theorem \ref{thm:relbd} that the corresponding two-graph is regular. Moreover, the maximum number of equiangular lines in a $15$-dimensional vector space
is $36$, so no other line of $\Omega$ is entirely contained in $V$. 
\end{proof}

Choose any set $\mathfrak{S}=\{\alpha_{i_j}\}_{j=1}^8$ of $8$ elements in the support of $x$ on $\mathfrak{B}$, i.e.~$\Gamma_2(\gamma)$, that are not pairwise interchanged by $x$. 
Then a basis for $W$, the $(-1)$-eigenspace of $x$, is $\mathfrak{B}_W=\{w_1,\ldots,w_8\}$ where
$$w_j=\alpha_{i_j}-\alpha^x_{i_j}\textnormal{ for $1\leq j\leq 8$ }.$$
Let $\E$ be the set of $240$ vectors in $\U$ that have a non-trivial component in $W$, that is, $\E$ is the set $\U$ minus the
$36$ vectors given in the proof of Proposition \ref{prop:36lines}.

\begin{Lem}\label{lem:sizedW} 
Let $x$ and $\E$ be as above. For
$\delta\in\E$ let $\delta_W$ be the projection of $\delta$ onto $W$.
Then $$(\delta_W,\delta_W)=\frac{2}{5}.$$
\end{Lem}

\begin{proof}
For each $\delta\in\E$,
$$\delta=\delta_W+\delta_V$$
for some $\delta_W\in W$ and $\delta_V\in V$. In fact,
$\delta_W=\frac{1}{2}(\delta-\delta^x)$ and $\delta_V=\frac{1}{2}(\delta+\delta^x)$. 

First suppose that $\delta=\alpha_i$ for some $\alpha_i\in\Gamma_2(\gamma)$. Then
$$\delta_W=\frac{1}{2}(\alpha_i-\alpha_i^x)=\pm\frac{1}{2}w_j$$
for some basis element $w_j\in\mathfrak{B}_W$. A straightforward calculation shows that $(\delta_W,\delta_W)=2/5$.

Now suppose that $\delta\in\E\backslash\Gamma_2(\gamma)$. Given that $\fix(x)=\Gamma_1(\gamma)$, we deduce from the expression \eqref{eq:vecsdef} for $\delta$ that 
\begin{equation}\label{eq:d-dx}\delta-\delta^x=\frac{1}{3}\sum_{\alpha_i\in\Gamma_2(\gamma)\cap\Gamma_1(\delta)}(\alpha_i-\alpha_i^x)-\frac{1}{6}\sum_{\alpha_i\in\Gamma_2(\gamma)\cap\Gamma_2(\delta)}(\alpha_i-\alpha_i^x).\end{equation}
If $\alpha_i,\alpha_i^x\in\Gamma_2(\gamma)\cap\Gamma_1(\delta)$ then $\alpha_i-\alpha_i^x$ and
$\alpha_i^x-\alpha_i$ both appear in the sum over $\Gamma_2(\gamma)\cap\Gamma_1(\delta)$ in \eqref{eq:d-dx} and hence cancel each other out. A similar statement 
holds if $\alpha_i,\alpha_i^x\in\Gamma_2(\gamma)\cap\Gamma_2(\delta)$. Therefore we only need to consider the elements 
$\alpha_i\in\Gamma_2(\gamma)\cap\Gamma_1(\delta)$ that are mapped by $x$ into $\Gamma_2(\delta)$. Equivalently,
if $\alpha_i$ is one of these elements then $\alpha_i^x\in\Gamma_2(\gamma)\cap\Gamma_2(\delta)$, and $\alpha_i^x$ is mapped by $x$ into $\Gamma_1(\delta)$. 
In particular, each such $\alpha_i$ contributes exactly $$\frac{1}{3}(\alpha_i-\alpha_i^x)-\frac{1}{6}(\alpha_i^x-\alpha_i)=\frac{1}{2}(\alpha_i-\alpha_i^x)=\pm\frac{1}{2}w_j$$
to $\delta-\delta^x$, for some basis element $w_j\in\mathfrak{B}_W$.

As $\gamma\in V$ and all vectors identified with the $28$ heptads of Type $1$ are wholly contained in $V$, $\Gamma_1(\delta)$ is of Type $2$ or $3$. It follows from Proposition \ref{prop:invtypes} that 
$|\Gamma_1(\delta)\cap\Gamma_1(\delta)^x|=3$, from which we deduce that exactly $4$ elements of $\Gamma_1(\delta)$ are mapped by $x$ into $\Gamma_2(\delta)$. Hence
$\delta_W=\frac{1}{2}(\delta-\delta^x)$ written in the basis $\mathfrak{B}_W$ is some permutation of the form vector
$$(\pm\frac{1}{4},\pm\frac{1}{4},\pm\frac{1}{4},\pm\frac{1}{4},0,0,0,0).$$
We conclude that $(\delta_W,\delta_W)=2/5$ as $(w_j,w_j)=8/5$ for $w_j\in\mathfrak{B}_W$.
\end{proof}

We are now in a position to prove that the projection of $\E$ onto $W$, the $(-1)$-eigenspace of $x$, gives a set of vectors that can be identified with the roots of the $E_8$ lattice.

\begin{Thm}\label{thm:rootsE8} Let $\E_8$ be the projection of $\E$ onto $W$. Then after scaling $\E_8$ appropriately, 
there is an orthogonal transformation of $W$ mapping $\E_8$ onto the minimal vectors of the $E_8$ lattice.
\end{Thm}

\begin{proof}
First we determine the size of $\E_8$. Suppose
that $\delta_W=\beta_W$ for some distinct $\delta,\beta\in\E$. Then
$$\delta-\delta^x=\beta-\beta^x.$$
This implies that $\{\delta,\delta^x,\beta,\beta^x\}$ is a linearly dependent set of four vectors. Hence the determinant of their Gram matrix
must equal zero. However, by the two-graph property, we know that by replacing vectors with their negative if necessary, the Gram matrix is equal to either
$\pm1/5(J-I)+I$, where $J$ is the all $1$ matrix and $I$ the identity matrix, or
\[\begin{pmatrix}
1&-1/5&1/5&1/5\\
-1/5&1&1/5&1/5\\
1/5&1/5&1&1/5\\
1/5&1/5&1/5&1
\end{pmatrix}\]
each of which has a non-zero determinant. Hence $|\E_8|=|\E|=240$. 

We now calculate the possible angles between $\delta_W$ and $\beta_W$. If $(\delta,\beta)=(\delta^x,\beta)$, one deduces that $(\delta_W,\beta_W)=0$,
so $\delta_W$ and $\beta_W$ are orthogonal. Otherwise
$(\delta,\beta)=-(\delta^x,\beta)$, which implies that $(\delta_V,\beta_V)=0$ and
$(\delta_W,\beta_W)=\pm1/5$. 
By Lemma \ref{lem:sizedW}, $||\delta_W||=||\beta_W||=\sqrt{2/5}$, and hence $\cos\theta=\pm1/2$
where $\theta$ is the angle between $\delta_W$ and $\beta_W$. 

Finally, observe that if $\delta_W\in\E_8$ then $-\delta_W\in\E_8$, since $x$ acts as $-I$ on $W$. Therefore, because $|\E_8|=240$, 
there exist $\delta_W,\beta_W\in\E_8$ such that $\cos\theta=1/2$ where $\theta$ is the angle between $\delta_W$ and $\beta_W$. 
Hence $\sqrt{5/2}\E_8$ is a $(8,240,1/2)$ spherical code. 
The result now follows from \cite[Theorems 5 \& 7, p.342-345]{conwaysloane}.
\end{proof}

\subsection{Proof of Theorem \ref{thm:main5}}\label{sec:pfmain5}
Before we prove this theorem, let us describe how one can identify the $276$ equiangular lines in $\R^{23}$ with certain antipodal lattice vectors in the Leech lattice $\Lambda$.
Using Conway's notation, let $\Lambda(n)$ denote the set of lattice vectors $v$ such that $v.v=16n$, and
let $v$ be any lattice vector contained in $\Lambda(3)$. Then by \cite[Theorem 27, p.288]{conwaysloane}, as $\Aut(\Lambda)$ acts transitively on $\Lambda(3)$, we
can assume without loss of generality that $v=(5,1^{23}).$ 
In \cite[Section 3.6, p.293]{conwaysloane}, Conway considers the stabiliser in $\Aut(\Lambda)$ of the lattice vector $v=(5,1^{23})$, which is equal to the sporadic simple group $Co_3$. 
Conway states that there are $276$ unordered pairs $\{y,z\}$ of lattice vectors of minimal norm such that $v=y+z$, and he proves that $Co_3$ has a $2$-transitive action on this set. 
Moreover, he states that there are 23 pairs where y (say) has the shape $(4^2,0^{22})$ and $253$ pairs with $y$ has the shape $(2^8,0^{16})$, with the first entry non-zero.
Now let $$\U_{552}=\{\pm(y-z)\,\, |\,\, v=y+z,\textnormal{ with $y,z\in\Lambda$ being minimal norm vectors}\}.$$
One can check that $y-z$ is orthogonal to $v$, and using the \cite[Theorem 25, p.287]{conwaysloane}, one deduces that $\U_{552}$ consists of $276$ antipodal lattice vectors in $\Lambda(5)$
that span $276$ equiangular lines in $\R^{23}$. Indeed this representation of the $276$ equiangular lines is the one given by Taylor \cite[Example 6.6]{taylor}.

\begin{proof}[Proof of Theorem \ref{thm:main5}] Let $\Lambda_{276}$ be the sublattice of $\Lambda$ generated by the elements of $\U_{552}$. Then the theorem is a consequence of Theorem \ref{thm:rootsE8}.
\end{proof}
\subsection{Maximal sets Equiangular lines in Lower Dimensions}
Given the above identification of the roots of the $E_8$ lattice with a subset of lines in $\Omega$, we can now identify various other maximal sets of equiangular lines in lower dimensions as 
either subsets of $\Omega$, or projections of subsets of $\Omega$ onto certain subspaces. Proposition \ref{prop:36lines} already identifies a subset of $36$ lines in $\Omega$ as a maximal set of
equiangular lines in $\R^{15}$. The next result allows us to identify the $28$ equiangular lines in $\R^7$ with a subset of $56$ lines in $\Omega$ projected onto a $7$-dimensional subspace.

\begin{Prop}\label{prop:28lines} Let $\alpha\in \E_8$, where $\E_8$ is as in Theorem \ref{thm:rootsE8}, and let $\N$ be the elements of $\E_8$ that are not orthogonal to $\alpha$. Then the projection of $\N$ onto
$\alpha^\perp$ (the orthogonal complement of $\alpha$ in $W$) is a set of $56$ vectors that span $28$ equiangular lines in $\R^7$.
\end{Prop}

\begin{proof}
First we replace $\E_8$ with $\sqrt{5/2}\E_8$ so that $(\alpha,\alpha)=1$ for all $\alpha\in\E_8$.  Now fix $\alpha\in\E_8$. 
It is known that $(\alpha,\beta)=\pm 1,\pm 1/2$ or $0$ for all $\beta\in\E_8$, and that 
$$A_1=1,\,\,A_{1/2}=56,\,\,A_0=126,\,\,A_{-1/2}=56,\,\,A_{-1}=1$$
where $A_i$ denote the number of elements of $\E_8$ whose inner product with $\alpha$ is equal to $i$ \cite[Theorem 5, p. 342]{conwaysloane}.

Let $U=\alpha^\perp$, the orthogonal complement of $\alpha$ in $W$, and let $\beta\in\E_8$ such that $(\alpha,\beta)=\pm1/2$. Then $\beta=\pm1/2\alpha+w$ for some $w\in U$
such that $(w,w)=3/4$. Furthermore, by \cite[Corollary 8.7]{bhall}, if 
$(\alpha,\beta)=\pm1/2$ then $$\alpha\mp\beta=\frac{1}{2}\alpha\mp w\in \E_8.$$
Let \begin{equation}\label{eq:28lines}S=\{w\in U\,|\,\frac{1}{2}\alpha+w\in \E_8\textnormal{ or }-\frac{1}{2}\alpha+w\in \E_8\}.\end{equation} Then we deduce $w\in S$ if and only if $-w\in S$, and in particular, $|S|=56$. 

Now let $w_1,w_2\in S$ with $w_1\neq -w_2$. As $(1/2\alpha+w_1,1/2\alpha+w_2)=0,\pm 1/2$ it follows that
$$(w_1,w_2)=\pm\frac{1}{4},\textnormal{ or $-\frac{3}{4}$}.$$ 
However, $$(w_1,w_2)=||w_1||||w_2||\cos\theta=\frac{3}{4}\cos\theta,$$ so $(w_1,w_2)=-3/4$ if and only if $w_1=-w_2$, a contradiction.
Thus $(w_1,w_2)=\pm1/4$ and we conclude that the vectors in $S$ span $28$ equiangular lines in $U\cong \R^7$ with common angle $\kappa=1/3$.
\end{proof}

It is known that the maximal set of equiangular lines in $\R^d$ for $d=5$ and $6$ can be found as subsets of
the $28$ equiangular lines in $\R^7$ (see for example \cite[Propositions 10.3.11 and 10.3.15]{green}). To see this explicitly, let $\Omega_7$ be the set of $28$ equiangular lines in
$\R^7$ described in Section \ref{sec:examples}, and consider any distinct pair $v(i),v(j)\in\Gamma$. Then
$$\Omega_6=\{v(k)\,|\,i,j\neq k\}\cup\{v(B)\,|\,B\cap\{i,j\}=\emptyset\textnormal{ or }B=\{i,j\}\}$$
is a set of $16$ equiangular lines in $\R^6$. Now let
$$\Gamma_6=\{v(k)\,|\,i,j\neq k\}\cup\{v(\{i.j\})\}.$$
Then $\Gamma_6$ is a maximal incoherent subset contained in $\Omega_6$ and $\Omega_6\backslash\Gamma_6$ is a set of $10$ equiangular lines in $\R^5$. 
Thus, we have the following result.

\begin{Thm}\label{thm:subsetlines} The maximal set of equiangular lines in $\R^d$, for $d=5,6,7,15$, appears as either a subset or a projection onto a subspace of a subset of the maximal set of 
$276$ equiangular lines in $\R^{23}$.
\end{Thm}

\subsection{Relationship with Exceptional Curves of Del Pezzo Surfaces}

A \emph{del Pezzo surface $\mathcal{S}$} is a smooth rational surface whose anticanonical divisor $-K_\mathcal{S}$ is ample.
The \emph{degree} of a del Pezzo surface is the self intersection number $D=(K_\mathcal{S},K_\mathcal{S})$, and
it is known that $1\leq D\leq 9$ \cite[Theorem 24.3]{manin}. It is also well known (\cite[Theorem 24.4]{manin}) that a del Pezzo surface of degree $D$ is 
either a product of two projective lines, in which case $D=8$, or the blow up of $9-D$ points in general position in the projective plane, 
where general position means no three points are collinear, no six lie on a conic, and no eight lie on a cubic having a double point at one of them.
We let $dP_D$ denote a del Pezzo of the second kind.

A curve $\C$ in $dP_D$ is a $(-1)$-curve (or exceptional curve) if it has self intersection number $-1$. The number of $(-1)$-curves contained in $dP_D$ 
is finite and given in Table \ref{tab:delpez} \cite{manin}. We now describe how the $(-1)$-curves of a del Pezzo surface $dP_D$ can be identified with certain sets of equiangular lines.

To do this we introduce the Gosset polytopes. These were first described by Gosset in \cite{gosset}, and subsequently studied further by Coxeter (see for example \cite{coxeter1,coxeter2}).  
Du Val \cite{duval} showed that the $(-1)$-curves of a del Pezzo surface $dP_D$ can be identified with the vertices of the Gosset polytope $(5-D)_{21}$ (here we are using Coxeter's notation 
for the polytopes \cite{coxeter2}).
The polytopes $(5-D)_{21}$, for $1\leq D\leq 8$, can all be constructed from the $4_{21}$ polytope. In particular, the polytope $(5-D)_{21}$ is equal to
the vertex figure of the polytope $(6-D)_{21}$ \cite{coxeter2}, so starting with the polytope $4_{21}$ one can consecutively construct the Gosset polytopes of interest to us.

It is well known that the vertices of the polytope $4_{21}$ can be identified with the $240$ roots of $E_8$ (see, for example, \cite{conwaysloane2}).
It is therefore a consequence of Theorem \ref{thm:rootsE8} that we can identify the vertices of the $4_{21}$ polytope, and hence the $(-1)$-curves of $dP_1$, with $240$ equiangular lines in $\R^{23}$ with common angle $\kappa=1/5$. Moreover, by Proposition \ref{prop:36lines},
one can add $36$ equiangular lines to these $240$ lines to give the maximal $276$ equiangular lines in $\R^{23}$. 
As the vertices of the polytopes $(5-D)_{21}$ can be identified with subsets of the vertices of $4_{21}$,
we can naturally identify these vertices with sets of equiangular lines in $\R^{23}$ with common angle $\kappa=1/5$. However, we can also make subsequent identifications with
equiangular lines with common angle $\kappa=1/3$.

\begin{table*}\centering
\ra{1.3}
\begin{tabular}{@{}lrrrrrrrr@{}}\toprule
Degree of $dP_D$&$1$&$2$&$3$&$4$&$5$&$6$&$7$&$8$\\ \midrule
No.~$(-1)$-curves&240&56&27&16&10&6&3&1\\
\bottomrule
\end{tabular}
\caption{Number of $(-1)$-curves of a del Pezzo surface $dP_D$.}\label{tab:delpez}
\end{table*}

By considering the vertex figure of $4_{21}$, the vertices of the polytope $3_{21}$ can be identified with the set of roots $\beta$ of $E_8$ such that $(\alpha,\beta)=1/2$ for any given root $\alpha$.
We saw in the proof of Proposition \ref{prop:28lines} that there are $56$ such roots. 
The centre of this polytope is the vertex $\alpha/2$. If we centre the polytope at the origin we find its vertices are the elements
in the set $S$ given in \eqref{eq:28lines}, which is equal to the projection of the $56$ roots $\beta$ onto the orthogonal complement of $\alpha$. Therefore the vertices of $3_{21}$ come
in antipodal pairs and generate a set $\Omega_7$ of $28$ equiangular lines with common angle $\kappa=1/3$ in $\R^7$ through the centre of the polytope. Hence each vertex of $3_{21}$ can 
be identified with line in $\Omega_7$, with antipodal vertices being identified with the same line. 
As the vertices of the polytope $(5-D)_{21}$, for $3\leq D\leq 8$, can be identified with subsets of the vertices of $3_{21}$, there is a natural correspondence between
the vertices of $(5-D)_{21}$ and subsets of $\Omega_7$. Furthermore, for $3\leq D\leq 8$, we find that the set of equiangular lines under this identification corresponds
to the two-graph of the intersection graph of the $(-1)$-curves of the del Pezzo surface $dP_D$.

Let us give more detail on this latter relationship. Given any line $\gamma\in\Omega_7$, the vertices of $2_{21}$ can be identified with the $27$ lines in $\Omega_7\backslash\{\gamma\}$.
These $27$ lines correspond to the two-graph of the complement of the Schl{\"a}fli graph, the unique strongly regular graph 
with parameters $(27,10,1,5)$ (\cite[Lemma 10.9.4]{godsilroyal}). This graph is also the intersection graph of $dP_3$ (see for example \cite{seid-3}).

Choose any other line $\delta\in\Omega_7\backslash\{\gamma\}$. Then
from the parameters of the two-graph corresponding to $\Omega_7$, we deduce that there are $16$ lines $\nu\in\Omega\backslash\{\gamma,\delta\}$ such that 
$$(\gamma,\delta)(\gamma,\nu)(\delta,\nu)>0.$$
The vertices identified with these $16$ lines are the vertices of the $1_{21}$ polytope.
Given that the automorphism group of $\Omega_7$ acts $2$-transitively, we deduce from the example preceding Theorem \ref{thm:subsetlines} 
these are the maximal set of $16$ equiangular lines in $\R^6$. Also, the corresponding two-graph is equal to the two-graph of the 
$(16,5,0,2)$ strongly regular Clebsch graph, the intersection graph of the $(-1)$-curves of $dP_4$ (see for example \cite{seid-3}).

Starting with $D=4$ and $\Omega_6$, the set of $16$ equiangular lines in $\R^6$, 
for $5\leq D\leq 7$ let $\Omega_{10-D}$ be the set $\Omega_{11-D}\backslash\Gamma_{11-D}$, where $\Gamma_{11-D}$ is a maximal incoherent subset contained in $\Omega_{11-D}$. 
Then for $4\leq D\leq 7$, the set $\Omega_{10-D}$ is a maximal set of equiangular lines in $\R^{10-D}$ with common angle $\kappa=1/3$ that we can identify with the vertices 
of the $(5-D)_{21}$ polytope. As we said before, the set $\Omega_{10-D}$ corresponds to the two-graph of the intersection graph of the $(-1)$-curves of $dP_{D}$, which
are give in Table \ref{tab:graphdelpez} (see \cite{manin} and \cite{seid-3}). 
It is reasonable to say that this relationship is degenerate for $dP_8$ because it has only one $(-1)$-curve, that is, the notion of equiangular 
for a single line is degenerate.

\begin{table*}\centering
\ra{1.3}
\begin{tabular}{@{}cc@{}}\toprule
Degree of $dP_D$&Intersection Graph\\ \midrule
3&Complement of Schl{\"a}fli\\
4&Clebsch\\
5&Peterson\\
6&Hexagon\\
7&
\tikzstyle{every node}=[circle, draw, fill=black!50,
                        inner sep=0pt, minimum width=4pt]
\begin{tikzpicture}[thick,scale=0.8]
    \draw{ (-1.5,0) node {}  -- (0,0) node{} -- (1.5,0) node {} };
\end{tikzpicture}\\
8&\tikzstyle{every node}=[circle, draw, fill=black!50,
                        inner sep=0pt, minimum width=4pt]
\begin{tikzpicture}[thick,scale=0.8]
    \draw{(0,0) node{} };
\end{tikzpicture}\\

\bottomrule
\end{tabular}
\caption{Intersection graphs of $dP_D$.}\label{tab:graphdelpez}
\end{table*}

To identify the vertices of $(5-(D+1))_{21}$ (or the $(-1)$-curves of $dP_{D+1}$) from the set $\Omega_{10-D}$ 
of equiangular lines, for $4\leq D\leq 6$, 
we had to find a maximal incoherent subset $\Gamma_{10-D}$ of lines in $\Omega_{10-D}$. In order to maintain the identification between equiangular lines and the $(-1)$-curves, this required 
 the set of $16$ equiangular lines in $\R^6$ to have a maximal incoherent subset of size $6$. This is suggestive as to why the set of $16$ equiangular lines in $\R^6$ satisfies 
 the hypothesis of Theorem \ref{thm:maineqquasi}. 
Of course, we did not follow this procedure to find the vertices of $3_{21}$, $2_{21}$ or $1_{21}$.
Therefore one is left to ask, what is the significance, if any, of "maximal incoherent subsets" of $(-1)$-curves in del Pezzo surfaces $dP_D$ for $4\leq D\leq 6$?

\section{On the Existence of Certain Quasi-symmetric designs}
By Theorem \ref{thm:maineqquasi}, the problem of classifying all sets of equiangular lines that saturate both the relative bound and the incoherence bound
is equivalent to classifying the quasi-symmetric designs with the parameters given in the theorem. By Remark \ref{rem:intermsints}, this is the same as the following two problems.

\begin{Prob}\label{prob:quasi1} Classify all $2$-$(d(s_1,s_2),k(s_1,s_2),\lambda(s_1,s_2);s_1,s_2)$ quasi-symmetric designs where $d(s_1,s_2)$, $k(s_1,s_2)$, and $\lambda(s_1,s_2)$
are given in Remark \ref{rem:intermsints}.
\end{Prob}

\begin{Prob}\label{prob:quasi2} Classify all $2$-$(2i(2i+1)-1,(2i-1)(i+1),i(2i^2+i-2);i^2+i-1,i^2-1)$ quasi-symmetric design, for integers $i\geq 1$. 
\end{Prob}

It is actually straightforward to show that we only need consider the designs in  Problem \ref{prob:quasi1} up to complementarity.

\begin{Lem}\label{lem:comp} The complement of a $(d(s_1,s_2),k(s_1,s_2);s_1,s_2)$ block set is a 
$(d(\hat{s}_1,\hat{s}_2),k(\hat{s}_1,\hat{s}_2);\hat{s}_1,\hat{s}_2)$ block set where
$$\hat{s}_1=d-2k+s_1,\,\,\,\hat{s}_2=d-2k+s_2$$
Moreover, the set of equiangular lines constructed from the complement block set, given in Construction \ref{const:1} and also Theorem \ref{thm:const1}, 
are equivalent under an orthogonal transformation to the set of equiangular lines constructed from the original block set.
\end{Lem}

\subsection{Necessary Conditions} Various necessary conditions have to be satisfied for quasi-symmetric designs with the parameters in the above problems to exist.
For example, for a design given in Problem \ref{prob:quasi1} to exist, the parameters given in Remark \ref{rem:intermsints} must be integers.
We also have, by \eqref{eq:fuse}, that
$$(\rho(s_1,s_2)^2-d(s_1,s_2)-2)\geq 0$$
which after substitution implies 
$$-(s_1-s_2)^4+2s_1^2s_2-4s_1s_2^2+2s_2^3+2s_1^2-2s_1s_2-s_2^2-s_1\geq 0.$$
By letting $m+s_2=s_1$, we find that
\begin{equation}\label{eq:int}\frac{2m(m+1)-1-\sqrt{\Delta}}{2} \leq s_2\leq\frac{2m(m+1)-1+\sqrt{\Delta}}{2}\end{equation}
where $$\Delta=(2m-1)(4m^2+6m-1).$$
Therefore, for any positive integer $m$, we can let $s_2$ be an integer in the above interval, $s_1=s_2+m$, and subsequently
construct parameters for possible quasi-symmetric designs, given the necessary existence conditions mentioned above. 
For example, if $m=1$, we deduce
that $0\leq s_2\leq 3$ and
$$(s_1,s_2)=(1,0),(2,1),(3,2)\textnormal{ or }(4,3).$$
No example exists for $(s_1,s_2)=(3,2)$ as $d(3,2)$ is not an integer. For $(s_1,s_2)=(1,0)$ or $(4,3)$, we get respectively, the $2$-$(7,2,1;1,0)$ quasi-symmetric design and its complement. These 
designs correspond to the maximal set of $28$ lines in $7$-dimensions. 
For $(s_1,s_2)=(2,1)$, we obtain the unique $2$-$(6,3,2;2,1)$ quasi-symmetric design, which corresponds to the 
maximal set of $16$ equiangular lines found in $\R^6$. This example illustrates the next theorem.

\begin{Thm} Let $i$ be a positive integer. Then a \begin{equation}\label{eq:desdeq2k}2-(2i(2i+1),i(2i+1),i(2i-1)(i+1);i^2+i,i^2)\end{equation} quasi-symmetric design exists only if a 
$$2-(2i(2i+1)-1,(2i-1)(i+1),i(2i^2+i-2);i^2+i-1,i^2-1)$$ quasi-symmetric design exists.
\end{Thm}

\begin{proof} Suppose that a quasi-symmetric design with the parameters in \eqref{eq:desdeq2k} exists. 
As $s_1=i^2+i$ and $s_2=i^2$, we have that $$d=d(s_1,s_2),\,\,k=k(s_1,s_2),\,\,\,\lambda=\lambda(s_1,s_2).$$
Thus, by Corollary \ref{cor:thmcon}, there exists a set $\Omega$ of equiangular lines in $\R^d$ that saturates the relative and the incoherence bound with $\rho=2i+1$. Let $\Gamma$
be an incoherent subset of $d$ lines in $\Omega$. 
As $d^2-(\rho-1)^2(\rho+d)=0$, it follows from Proposition \ref{prop:tayeq} that $|\Gamma_1(\gamma)|=d/2$ for all $\gamma\in\Omega\backslash\Gamma$. The
result now follows from Theorem \ref{thm:deq2k}.
\end{proof}

Other necessary conditions apply, for example integrality conditions of the eigenvalues of the
adjacency matrix of the corresponding strongly regular graph (see for example \cite[Theorem 3.8]{shriksane}). The results of Calderbank in \cite{calderbank} also apply,
of which we partly mentioned in Theorem \ref{thm:calder}. More significantly for our purposes, Calderbank also proved the following result.

\begin{Thm}\cite[Theorem A]{calder1} Let $(\P,\B)$ be a $2$-$(v,k,\lambda)$ design with intersection numbers $s_1\equiv s_2\equiv\ldots\equiv s_\ell\equiv s (\mod 2)$. Then either
\begin{itemize}
\item[1)] $r\equiv \lambda\mod 4$;
\item[2)] $s\equiv 0\mod 2$, $k\equiv 0\mod 4$, $v\equiv\pm 1\mod 8$; or
\item[3)] $s\equiv 1\mod 2$, $k\equiv v\mod 4$, $v\equiv\pm 1\mod 8$.
\end{itemize}
\end{Thm}

Calderbank proved a similar theorem for the case that $p$ is an odd prime \cite[Theorem 2]{calder2}, and then generalised this further in work with Blokhuis \cite{calder3}. The statements
of these results are more involved, so we refer the reader to the original papers.

\subsection{Possible parameter sets} We have constructed a table of parameters of possible quasi-symmetric designs that satisfy the conditions of Problem \ref{prob:quasi1}.
These can be found in Table \ref{tab:mset}. For a given integer $1\leq m\leq 10$, we first considered all pairs $(s_1,s_2)$ where $s_1=m+s_2$ and $s_2$ in an integer in the interval \eqref{eq:int}. 
We then discarded examples that do not satisfy the various necessary integral conditions for the parameters of the design. One will see in the last column, labelled ``Existence", 
that we have been able to discard some of these remaining parameter sets, mainly as a consequence of the theorems of Calderbank mentioned above. 
Upon examination of Table \ref{tab:mset}, three families of parameters seem worth highlighting. For a positive integer $i$, the parameters for
each family can be found in one of the columns of Table \ref{tab:params}. In Table \ref{tab:iset} we have recorded the parameters of the (possible) quasi-symmetric designs, and corresponding sets of lines, for the first 10 values of $i$ for each family.

\begin{table}\centering
\ra{1.3}
\begin{tabular}{@{}r|lll@{}}\toprule
Family&$1$&$2$&$3$\\ \midrule
$d$&$i(i^3+6i^2+11i+5)$&$2i(2i+1)$&$(4i^2+4i-1)(i^2+i-1)$\\
$k$&$\frac{1}{2}i(i^3+5i^2+7i+1)$&$i(2i+1)$&$(2i-1)(i+1)(i^2+i-1)$\\
$\lambda$&$\frac{1}{4}i(i+2)(i^2+2i-1)(i^3+5i^2+7i+1)$&$i(2i-1)(i+1)$&$(2i-1)(i^2+i-1)(2i^3+3i^2-2i-2)$\\
$s_1$&$\frac{1}{4}i(i+2)(i+1)^2$&$i^2+i$&$i^2(i^2+i-1)$\\
$s_2$&$\frac{1}{4}i(i^3+4i^2+3i-4)$&$i^2$&$(i^2-1)(i^2+i-1)$\\ \midrule
$m$&$\frac{1}{2}i(i+3)$&$i$&$i^2+i-1$\\
$r$&$\frac{1}{2}i(i^3+5i^2+6i-1)(i^3+5i^2+7i+1)$&$i(4i^2+2i-1)$&$(2i-1)(i+1)(4i^2+4i-5)(i^2+i-1)$\\
$r-\lambda$&$\frac{1}{4}i^2(i+3)^2(i^3+5i^2+7i+1)$&$i^2(2i+1)$&$(2i-1)(2i+3)(i^2+i-1)^2$\\ \midrule
$|\Omega|$&$i^2(i+2)(i+3)(i^3+6i^2+11i+5)$&$8i^2(i+1)$&$4(i^2+i-1)^2(4i^2+4i-1)$\\
$\rho$&$i^2+3i+1$&$2i+1$&$2i^2+2i-1$\\
$a$&$\frac{1}{2}i^2(i+3)^2(i^3+5i^2+7i+1)$&$2i^2(2i+1)$&$2(2i-1)(2i+3)(i^2+i-1)^2$\\
\bottomrule
\end{tabular}
\caption{Parameters of (possible) families quasi-symmetric designs and sets of equiangular lines that saturate
the incoherence bound, where $i$ is a positive integer.}\label{tab:params}
\end{table}

An interesting observation is that for each family of parameters in Table \ref{tab:params},
$$r-\lambda=m^2f(i)$$
for some polynomial $f(x)$, where $m=s_1-s_2$. 
Thus for each prime power $p^e$ such that $s_1\equiv s_2 \mod p^e$, $r\equiv \lambda\mod p^{2e}$.
This means that the necessary conditions of Calderbank in \cite{calder1, calder2} and of Blokhuis and Calderbank in \cite{calder3} are satisfied.
However, we can eliminate some of the parameter sets using the results of Calderbank and Frankl in \cite{calder4}.

\begin{Prop} A quasi-symmetric design with the parameters of Family $1$ (Family $2$) in Table \ref{tab:params} does not exist if $i\equiv 4\mod 8$ ($i\equiv 2\mod 4$).
\end{Prop}

\begin{proof} One calculates for the parameters of Family $1$ (Family $2$) that $s_1\equiv s_2 \mod 2$ and $k\equiv 2\mod 4$ 
if and only if $i\equiv 4\mod 8$ ($i\equiv 2\mod 4$). Thus the results in \cite{calder4} can be applied for these parameter sets. In both cases,
one goes through the congruence conditions of \cite[Lemma 2 (B)]{calder4} and finds that none of them hold. Hence we deduce from \cite[Theorem 3]{calder4} that a quasi-symmetric design 
with these parameters does not exist.
\end{proof}

\begin{Cor} For $i\equiv 4\mod 8$, if there exists a set $\Omega$ of equiangular lines with common angle $\kappa=1/(i^2+3i+1)$ in $\R^d$ that saturates the relative
bound, where $d=i(i^3 +6i^2 +11i+5)$, then $\Inc(\Omega)<d$.
\end{Cor}

Since $s_1\not\equiv s_2 \mod 2$ for the set of parameters that appear in Family $3$, we cannot use the results of Calderbank and Frankl in \cite{calder4} to eliminate 
any members from this family. In particular, we have not been able to eliminate any members of this family.

Finally, for the design in Problem \ref{prob:quasi2}, $s_1\equiv s_2 \mod 2$ if and only if $i\equiv 0 \mod 2$ if and only if $k\equiv 1\mod 2$. Therefore we cannot use the results of Calderbank and Frankl to eliminate any member from this family. Moreover
$$r-\lambda=i^2f(i),$$
for some polynomial $f(x)$, where $s_1-s_2=i$, so as above, the necessary conditions of Calderbank in \cite{calder1, calder2} and of Blokhuis and Calderbank in \cite{calder3} are satisfied. Indeed, the only member
of this family that has been proven not to exist occurs when $i=2$; Calderbank \cite[Theorem 13]{calder1} used certain properties of the $[24,12, 8]$ binary Golay code to prove this.

\begin{Rem}
We highlight one more family of quasi-symmetric designs that satisfy the conditions of Problem \ref{prob:quasi1}. For a positive integer $i$, let 
$$s_1=i^2,\,\,\,\, s_2=i(i-1)$$
so $$d(s_1,s_2)=4i^2+2i+1,\,\,\,\,k(s_1,s_2)=2i^2,\,\,\,\,\lambda(s_1,s_2)=i(2i^2-1).$$ 
The parameters of the corresponding set of equiangular lines are 
$$\rho=2i+1,\,\,\,\,|\Omega|=\rho^3+1,\,\,\,\,a=(\rho-1)(\rho^2+1)/2.$$
For this family one can use the results of Calderbank above to show that for certain values of $i$ the quasi-symmetric design does not exist, for example $i=2,3,6,7,10$. However,
for infinitely many values of $i$, there exists a set of equiangular lines with these parameters that saturate the relative bound (see \cite[Section 10]{seidelsurvey}), in particular when $i=(p^e-1)/2$ for some odd prime 
$p$ and integer $e$. These examples highlight
the fact that even if a set of equiangular lines that saturates the relative {\bf{and}} incoherence bounds does not exist, a set of lines 
which saturates only the relative bound with the same parameters may exist. With this in mind, 
the proof of Theorem \ref{thm:main} suggests it may be possible to find a set of equiangular lines in dimension $d=839$ that saturate the absolute bound. Of course this would invalidate Conjecture \ref{conj:abs}, so a resolution of this either way would be useful.
\end{Rem}

\begin{Prob} Does there exist a set of equiangular lines in $\R^{839}$ that saturates the absolute bound?
\end{Prob}


\appendix

\section{Tables}

\begin{table}\centering
\ra{1.3}
\begin{tabular}{@{}rrrrrrrrrrc@{}}\toprule
 \multicolumn{5}{c}{QS-Design} & \phantom{abc}& \multicolumn{3}{c}{EQ-lines}& \phantom{a}&\multicolumn{1}{c}{Existence}\\
\cmidrule{1-5} \cmidrule{7-9} \cmidrule{11-11} 
$d$ & $k$ & $\lambda$ & $s_1$ & $s_2$ && $|\Omega|$ & $\rho$ & $a$ && Yes/No/?\\ \midrule
6&3&2&2&1&&16&3&6&&Yes\\
7&2&1&1&0&&28&3&10&&Yes\\
20&10&18&6&4&&96&5&40&&No\\
21&8&14&4&2&&126&5&52&&No\\
23&7&21&3&1&&276&5&112&&Yes\\
42&21&60&12&9&&288&7&126&&?\\
43&18&51&9&6&&344&7&150&&No\\
72&36&140&20&16&&640&9&288&&?\\
73&32&124&16&12&&730&9&328&&?\\
110&55&270&30&25&&1200&11&550&&?\\
111&50&245&25&20&&1332&11&610&&?\\
115&45&330&20&15&&2300&11&1050&&?\\
118&43&602&18&13&&4720&11&2150&&?\\
156&78&462&42&36&&2016&13&936&&No\\
157&72&426&36&30&&2198&13&1020&&No\\
163&64&672&28&22&&4564&13&2112&&No\\
210&105&728&56&49&&3136&15&1470&&?\\
211&98&679&49&42&&3376&15&1582&&No\\
272&136&1080&72&64&&4608&17&2176&&?\\
273&128&1016&64&56&&4914&17&2320&&?\\
342&171&1530&90&81&&6480&19&3078&&?\\
343&162&1449&81&72&&6860&19&3258&&?\\
357&141&4935&60&51&&32130&19&15228&&?\\
420&210&2090&110&100&&8800&21&4200&&No\\
421&200&1990&100&90&&9262&21&4420&&No\\
\bottomrule
\end{tabular}
\caption{Parameters of (possible) quasi-symmetric designs and equiangular lines that saturate the incoherence bound for $1\leq s_1-s_2\leq 10$.}\label{tab:mset}
\end{table}

\begin{table}\centering
\ra{1.3}
\begin{tabular}{@{}lrrrrrrrrrrrc@{}}\toprule
&& \multicolumn{5}{c}{QS-Design} & \phantom{abc}& \multicolumn{3}{c}{EQ-lines}& \phantom{a}&\multicolumn{1}{c}{Existence}\\
\cmidrule{3-7} \cmidrule{9-11} \cmidrule{13-13} 
Family&& $d$ & $k$ & $\lambda$ & $s_1$ & $s_2$ && $|\Omega|$ & $\rho$ & $a$ && Yes/No/?\\ \midrule
$1$&$i=1$&23&7&21&3&1&&276&5&112&&Yes\\
&$2$ &118&43& 602 &18&13&&4720&11&2150&&? \\
&$3$ & 357 & 141 &4935 &60&51&&32130&19&15228&&?\\
&$4$ &  836& 346& 23874& 150& 136 && 140448& 29& 67816&&No\\
&5& 1675& 715& 85085& 315& 295 && 469000& 41& 228800&&? \\
&6&  3018& 1317& 247596& 588& 561 && 1303776& 55& 640062&&? \\
&7&  5033& 2233& 623007& 1008& 973 && 3170790& 71& 1563100&&? \\
&8&  7912& 3556& 1404620& 1620& 1576 && 6962560& 89& 3442208&&? \\
&9&  11871& 5391& 2905749& 2475& 2421 && 14102748& 109& 6986736&&? \\
&10&  17150& 7855& 5608470& 3630& 3565 && 26754000& 131& 13274950&&? \\ \midrule
2&$i=1$& 6&3&2&2&1 && 16&3&6&&Yes\\
&2& 20&10&18&6&4 && 96&5&40&&No\\
&3& 42&21&60&12&9 && 288&7&126&&?\\
&4& 72&36&140&20&16 && 640&9&288&&?\\
&5& 110&55&270&30&25 && 1200&11&550&&?\\
&6& 156&78&462&42&36 && 2016&13&936&&No\\
&7& 210&105&728&56&49 && 3136&15&1470&&?\\
&8& 272&136&1080&72&64 && 4608&17&2176&&?\\
&9& 342&171&1530&90&81 && 6480&19&3078&&?\\
&10& 420&210&2090&110&100 && 8800&21&4200&&No\\ \midrule
$3$&$i=1$&7&2&1&1&0&&28&3&10&&Yes\\
&2&115&45&330&20&15&&2300&11&1050&&?\\
&3&517&220&4015&99&88&&22748&23&10890&&?\\
&4&1501&665&22078&304&285&&114076&39&55594&&?\\
&5&3451&1566&81693&725&696&&400316&59&196794&&?\\
&6&6847&3157&237226&1476&1435&&1122908&83&554730&&?\\
&7&12265&5720&584155&2695&2640&&2698300&111&1337050&&?\\
&8&20377&9585&1275870&4544&4473&&5787068&143&2873370&&?\\
&9&31951&15130&2543353&7209&7120&&11374556&179&5655594&&?\\
&10&47851&22781&4717738&10900&10791&&20863036&219&10383994&&?\\
\bottomrule
\end{tabular}
\caption{Parameters of (possible) quasi-symmetric designs and equiangular lines of first $10$ members of Families in Table \ref{tab:params}.}\label{tab:iset}
\end{table}

\end{document}